\documentclass[10pt]{article}

\usepackage{amsfonts}
\usepackage{amssymb}
\usepackage{amsthm}
\usepackage{amsmath}
\usepackage{bm}
\usepackage{amscd} 
\usepackage{graphicx}

\usepackage{mathrsfs}
\usepackage{MnSymbol}

\usepackage{tikz-cd}

\usepackage{hyperref}

\numberwithin{equation}{section} \DeclareMathSizes{2}{10}{12}{13}
\input xy
\xyoption{all}

\makeatletter
\newcommand{\bigperp}{%
  \mathop{\mathpalette\bigp@rp\relax}%
  \displaylimits
}

\newcommand{\bigp@rp}[2]{%
  \vcenter{
    \m@th\hbox{\scalebox{\ifx#1\displaystyle1.2\else1.5\fi}{$#1\perp$}}
  }%
}

\newtheorem{thm}{Proposition}[section]

\newtheorem{Thm}[thm]{Theorem}
\newtheorem{rem}[thm]{Remark}

\newtheorem{cor}[thm]{Corollary}
\newtheorem{lem}[thm]{Lemma}
\newtheorem{defn}[thm]{Definition}

\title{Homological Algebra of Heyting modules}

\smallskip 

\smallskip 
\author{Abhishek Banerjee  \footnote{AB was partially supported by SERB Matrics fellowship MTR/2017/000112}}
\date{ }

\begin{document}

\maketitle

\centerline{\emph{Dept. of Mathematics, Indian Institute of Science, Bangalore-560012, India.}}
\centerline{\emph{Email: abhishekbanerjee1313@gmail.com}}

\smallskip

\begin{abstract} The collection of open sets of a topological space forms a Heyting algebra, which leads to the idea of a Heyting algebra as a generalized topological space. In fact, a sober topological space may be reconstructed from its locale of open sets. This has given rise to a good theory of presheaves and sheaves over locales. At the same time, several ring like properties of Heyting algebras have also been studied. The purpose of this paper is to study a non-abelian homological theory for modules over Heyting algebras. 
\end{abstract}

\smallskip

\smallskip
\emph{MSC (2010) Subject Classification: 18G50}  


\section{Introduction}

The essence of an abelian category is the fact that every morphism may be factored uniquely as a cokernel followed by a kernel. However, several difficulties arise when one tries
to adapt this axiom to non-abelian or even non-additive settings. This is the reason the study of homological algebra has been dominated by that of abelian categories. 

\smallskip
Over the years, many efforts have been made to develop homological algebra with categories that are not additive. One of the earliest appears to be a series of papers by Fr$\ddot{\mbox{o}}$hlich \cite{Fro1,Fro2,Fro3} which gave a theory of derived functors for non-abelian groups.  The work of Eckmann and Hilton \cite{EH1}, \cite{EH2}, \cite{EH3} was motivated by the category of topological spaces while Gerstenhaber gave a theory of Baer extensions in \cite{MG}. We also refer the reader to the study of Barr-exact categories (see \cite{Barr}), modular categories (see Carboni \cite{Carboni}), protomodular categories (see Bourn \cite{Bourn}), Mal'cev categories (see Carboni, Lambek and Pedicchio \cite{CLP}) semi-abelian categories (see Janelidze, M\'{a}rki and Tholen\cite{GJ}) and homological categories (see Grandis \cite{Grandis}) for other instances of work in this area. 

\smallskip
More recently, Connes and Consani \cite{one} have presented a comprehensive theory of ``Homological Algebra in characteristic one'' over the Boolean semifield $\mathbb B=\{0,1\}$. Their motivation was to develop a theory that would be  suitable for treating certain non-additive categories of   sheaves over a topos. In this paper, our objective is to develop a homological
theory similar to \cite{one} for modules over a Heyting algebra.

\smallskip
We recall that a  Heyting algebra $H$ is a lattice such that for each $x\in H$, the functor $\_\_\wedge x: H\longrightarrow H$ has a right adjoint. In other words, for each $x,y\in H$, there
is an element $(x\rightarrow y)_H\in H$ such that $z\leq (x\rightarrow y)_H$ if and only if $z\wedge x\leq y$. Here the partially ordered set underlying the lattice $H$ is treated as a category. For each element $x\in H$, we may consider $\righthalfcap x:=(x\rightarrow 0)_H$. Then, if $\righthalfcap \righthalfcap x=x$ for every $x\in H$, then $H$ becomes a Boolean algebra. In general, the collection of Heyting algebras contains the collection of Boolean algebras. 

\smallskip
The significance of Heyting algebras lies in the fact that they may be treated as generalized topological spaces, a deep idea that was first advanced by Ehresmann \cite{Ehres} and B\'{e}nabou \cite{Ben}. If $X$ is a topological space,  the collection $\Omega X$ of open sets of $X$ is a Heyting algebra : the joins are given by unions, the finite meets are given by interesection and 
\begin{equation*}
(U\rightarrow V)_{\Omega X}:=int(U^c\cup V)
\end{equation*} for any open subsets $U$, $V\in \Omega X$. We observe that $\Omega X$ need not be a Boolean algebra. For instance, if $X=\mathbb R$ with the usual topology, the open set $U=\mathbb R\backslash \{0\}$ satisfies $\righthalfcap U=int(U^c\cup \phi)=\phi$ and hence $\righthalfcap\righthalfcap U\ne U$. For any topological space $X$, the Heyting algebra
$\Omega X$ forms a locale  and there is an equivalence of categories between spatial locales and sober topological spaces (see \cite[Chapter II]{PJT}).  We recall that a topological space is said to be sober if every irreducible closed subset is the closure of a unique generic point. This applies in particular to the Zariski spectrum of a commutative ring. 

\smallskip
This treatment of Heyting algebras as generalized topological spaces has led to a well-developed theory of sheaves over locales, for which we again refer the reader to Johnstone's book 
\cite[Chapter V]{PJT}. At the same time, several ``ring like properties'' have also been studied for Heyting algebras, such as prime ideals, filters and maximal ideal theorem (see \cite[Chapter I]{PJT}). Therefore, it is only natural to study a theory of modules over Heyting algebras, which is the purpose of this paper.

\smallskip
We now describe the paper in more detail. We begin by introducing Boolean modules in Section 2. A Boolean module over a Boolean algebra $B$ is a join semilattice $M$ equipped with
an action $B\times M\longrightarrow M$ that is distributive and an operator $\righthalfcap_M:M\longrightarrow B$ satisfying certain conditions described in Definition \ref{Bmodule}. A key fact proved in this section is that when $B$ is a finite Boolean algebra, any distributive module over the lattice underlying $B$ is canonically equipped with such an operator $\righthalfcap_M:M\longrightarrow B$. This is in preparation for the notion of a Heyting module, which we introduce in Section 3. 

\smallskip
Let $H$ be a Heyting algebra. A Heyting module $M$ over $H$ is a join semilattice $M$ with an action $H\times M\longrightarrow M$ such that for each $m\in M$, the functor
$\_\_ \wedge m:H\longrightarrow M$ has a right adjoint (see Definition \ref{DD3.2}). In other words, for every $m$, $n\in M$, there is an element $(m\rightarrow n)_M\in H$ such that 
\begin{equation*}
x\wedge m\leq n \qquad\Leftrightarrow\qquad x\leq (m\rightarrow n)_M
\end{equation*} for any $x\in H$. We denote by $Heymod_H$ the category of Heyting modules over a Heyting algebra $H$. 
The basic properties of the operation $(\_\_\rightarrow \_\_)_M$ as well as those of morphisms of Heyting modules are established in Section 3.  In particular, if $H$ is a finite Heyting algebra, we show that any module $M$ over the underlying lattice $H$ is canonically equipped with the structure of a Heyting module.  We recall that for any Heyting algebra, the collection $H_{\righthalfcap\righthalfcap}$ of elements $x\in H$ satisfying $\righthalfcap \righthalfcap x=x$ forms a Boolean algebra. If $H_{\righthalfcap\righthalfcap}$ is a sublattice of $H$ and
$M\in Heymod_H$, we show that the collection $M_{\righthalfcap\righthalfcap}$ of all elements of $M$ such that $(m\rightarrow 0)_M\in H_{\righthalfcap\righthalfcap}$ forms a Boolean module over the Boolean algebra $H_{\righthalfcap\righthalfcap}$. 

\smallskip
In Section 4, we begin studying the dual $M^\bigstar=Heymod_H(M,H)$ of a Heyting module $M\in Heymod_H$. In order to understand a morphism $\phi : M\longrightarrow H$
in $Heymod_H$, we need to consider the submodules $K_c:=\{\mbox{$m\in M$ $\vert$ $\phi(m)\leq c$ }\}$ as $c$ varies over all elements of $H$. The relations between the objects
$\{K_c\}_{c\in H}$ are then formalized to define ``hereditary systems of submodules'' (see Definition \ref{D4.3}). If $H$ is a finite Heyting algebra, we obtain a one-one correspondence between
elements $\phi\in M^\bigstar$ and hereditary systems of submodules of $M$. Further, if $H$ is a  Boolean algebra,  hereditary systems of submodules correspond to ordinary
hereditary submodules of the join semilattice $M$. In the latter case, we are able to obtain an analogue of Hahn-Banach Theorem, which shows that the morphisms
in the dual $M^\bigstar$ are able to separate elements of $M$. 

\smallskip
It is easy to show that $Heymod_H$ is a semiadditive category, i.e., finite products and finite coproducts coincide. However, we would like to know the extent to which $Heymod_H$ satisfies a non-additive version of the (AB2) axiom
for abelian categories. In other words, we would like to compare coimages and images in $Heymod_H$. This is done in Section 5 by considering ``kernel pairs'' and ``cokernel pairs'' in a manner
analogous to \cite{one}. Additionally, for a submodule $K\subseteq N$, we set
\begin{equation*}
\widetilde{K}=\{\mbox{$n\in N$ $\vert$ $t_1(n)=t_2(n)$ $\forall$ $Q\in Heymod_H$ $\forall$  $t_1,t_2:N\longrightarrow Q$ such that
$t_1(k)=t_2(k)$ $\forall$ $k\in K$ }\}
\end{equation*} Then, if $H$ is a finite Heyting algebra and $f:M\longrightarrow N$ is a morphism in $Heymod_H$, we show that $\widetilde{Coim(f)}=Im(f)$. If $H$ is finite and Boolean,
this reduces to give $Coim(f)=Im(f)$. 

\smallskip
We continue in Section 6 with $H$ being a finite Heyting algebra. We show that $Heymod_H$ is a tensor category, with the functor $\_\_\otimes_HN:Heymod_H
\longrightarrow Heymod_H$ being left adjoint to $Heymod_H(N,\_\_):Heymod_H\longrightarrow Heymod_H$. This gives rise to a good theory of `extension and restriction of scalars'
for Heyting modules.

\smallskip
We now consider the endofunctor $\bigperp : Heymod_H\longrightarrow Heymod_H$ given by taking each $M\in Heymod_H$ to $M^2$ and each morphism $f:M\longrightarrow N$
to $(f,f):M^2\longrightarrow N^2$. This defines a comonad on $Heymod_H$ and we consider the corresponding Kleisli category 
$Heymod_H^2$ in Section 7 whereas the corresponding Eilenberg-Moore category $Heymod_H^{\mathfrak s}$ is studied in Section 8. For modules over the semifield $\{0,1\}$, the 
Kleisli and Eilenberg-Moore categories of the squaring endofunctor have been studied by Connes and Consani \cite{one}. The category $Heymod_H^2$ has the same objects as 
$Heymod_H$. A morphism $M\longrightarrow N$ in $Heymod_H^2$ consists of a pair $(f,g):M\longrightarrow N$ of morphisms in $Heymod_H$, with composition given by (see 
\eqref{composition})
\begin{equation*}
(f',g')\circ (f,g)=(f'\circ f + g'\circ g,f'\circ g+g'\circ f)
\end{equation*} On the other hand, an object of $Heymod_H^{\mathfrak s}$ is a pair $(M,\sigma_M)$ consisting of $M\in Heymod_H$ and $\sigma_M$ an $H$-linear involution of $M$. The conditions describing monomorphisms, epimorphisms and strict exactness for sequences in $Heymod_H^2$ and $Heymod_H^{\mathfrak s}$ are studied respectively in Sections 7 and 8. Further, we show in Section 9 that $Heymod_H^{\mathfrak s}$ is a semiexact category in the sense of Grandis \cite{Grandis} (see Theorem \ref{Tt9.6}). Finally, when $H$ is a finite Boolean algebra, we show
in Section 10 that $Heymod_H^{\mathfrak s}$ also becomes a homological category in the sense of Grandis \cite{Grandis} (see Theorem \ref{Tq10.10}). 

\smallskip
We return to the case of an arbitrary (not necessarily finite) Heyting algebra $H$ in Section 11. Using an ultrafilter criterion of Finocchiaro \cite{Fino} and some recent methods
of Finocchiaro, Fontana and Spirito \cite{FFS}, we show that the collection $Sub(M)$ of Heyting submodules of a Heyting module $M$ carries the structure of a spectral space. In other words, $Sub(M)$ is homeomorphic to the Zariski spectrum of a commutative ring. We show more generally that any $Sub^{\mathbf c}(M)$ is a spectral space, where  $Sub^{\mathbf c}(M)$  is the collection of submodules fixed by a closure operator $\mathbf c$ of finite type on $Sub(M)$ (see Definition \ref{D10.2} and Proposition \ref{P10.4}). In particular, it follows that the collection of hereditary submodules of a Heyting module $M$ is equipped with the structure of a spectral space. 

\section{Boolean modules}

\smallskip
By definition, a lattice $L$ is a partially ordered set $(L,\leq)$ such that every finite subset $S\subseteq L$ of elements 
of $L$ has a supremum (called the join $\vee S:=\underset{s\in S}{\vee}s$) and an infimum (called the meet $\wedge S:=\underset{s\in S}{\wedge}s$). Considering respectively the join and the meet of the empty subset, it follows that $L$ contains two distinguished elements $0$ and $1$ such that every $x\in L$ satisfies $0\leq x\leq 1$. 

\smallskip
A lattice is said to be distributive if it satisfies the following two conditions:

(1) For any $x$, $y$, $z\in L$, we have $x\wedge (y\vee z)=(x\wedge y)\vee (x\wedge z)$. 

(2) For any $x$, $y$, $z\in L$, we have $x\vee (y\wedge z)=(x\vee y)\wedge (x\vee z)$. 

\smallskip
In fact, it may be shown (see \cite[Lemma I.1.5]{PJT}) that a lattice $L$ satisfies  condition (1) if and only if it satisfies condition (2). Further, in a distributive lattice, it may be verified that given an element $x\in L$, there exists at most one element 
$y\in L$ such that $x\vee y=1$ and $x\wedge y=0$. Such an element (if it exists) is referred to as a complement of $x$. 

\begin{defn}\label{D2.0} Let $L$ be a distributive lattice. A module over $L$ consists of the following data:

\smallskip
(1) A join semilattice $M$, i.e., a partially ordered set $M$ such that every finite subset of $M$ has a join in $M$. 

\smallskip
(2) An action $f:L\times M\longrightarrow M$ that is order preserving in both variables and such that 
\begin{equation*}
\begin{array}{c}
f(1_L,m)=m\quad f(0_L,m)=0_M \quad f(x\wedge y,m)=f(x,f(y,m))\\
\end{array}
\end{equation*} for any $x$, $y\in L$ and $m\in M$. Here $0_L$ and $1_L$ are respectively the least element and largest element of 
$L$, while $0_M$ is the least element of the join semilattice $M$. For $x\in L$ and $m\in M$, we will typically denote 
by $x\wedge m$ the element $f(x,m)\in M$. 

\smallskip
Further, we will say that $M$ is a distributive module over $L$ if it satisfies
\begin{equation*}
f(x\vee y,m)=f(x,m)\vee f(y,m)\quad f(x,m\vee n)=f(x,m)\vee f(x,n)
\end{equation*}  for any $x$, $y\in L$ and $m$, $n\in M$.

\smallskip
A morphism $g:M\longrightarrow M'$ of modules over $L$ is a morphism of join semilattices such that $g(x\wedge m)=x\wedge g(m)$
for each $x\in L$, $m\in M$.
\end{defn}

\smallskip
\begin{defn}\label{D2.1}  (see \cite[$\S$ I.1.6]{PJT}) A Boolean algebra $B$ is a distributive lattice such that every element
$x\in B$ has a complement in $B$. This gives a unary operation on $B$,  which we will denote by $\righthalfcap_B:B
\longrightarrow B$.
\end{defn} 

It is clear that $\righthalfcap_B\circ\righthalfcap_B=id$ and $\righthalfcap_B(0)=1$. It can also be checked easily that $\righthalfcap_B:B
\longrightarrow B$ is order-reversing and satisfies De Morgan's laws:
\begin{equation}
\righthalfcap_B(x\vee y)=\righthalfcap_B(x)\wedge \righthalfcap_B(y)\qquad 
\righthalfcap_B(x\wedge y)=\righthalfcap_B(x)\vee \righthalfcap_B(y)
\end{equation} If $(B,\righthalfcap_B)$ and $(B',\righthalfcap_{B'})$ are Boolean algebras, the uniqueness of complements shows that a lattice morphism
$f:B\longrightarrow B'$ automatically satisfies $f(\righthalfcap_{B}(x))=\righthalfcap_{B'}(f(x))$ for each $x\in B$. We should also mention (see \cite[$\S$ 1.9]{PJT}) that the category of Boolean algebras is isomorphic to the category of Boolean rings. 

\smallskip
We will now introduce the concept of a Boolean module over a Boolean algebra. 

\begin{defn}\label{Bmodule} Let $(B,\righthalfcap_B)$ be a Boolean algebra. A Boolean module over  $(B,\righthalfcap_B)$ consists of the following data:

\smallskip
(1) A distributive module $M$ over the distributive lattice underlying $B$. 

\smallskip
(2) An operation $\righthalfcap_M:M\longrightarrow B$ satisfying the following conditions:
\begin{equation}\label{eq2.2}
\righthalfcap_M(0_M)=1_B \quad \righthalfcap_M(x\wedge m)=\righthalfcap_B(x)\vee \righthalfcap_M(m)\quad
\righthalfcap_M(m\vee n)=\righthalfcap_M(m)\wedge \righthalfcap_M(n)
\end{equation} for any $x\in B$ and $m$, $n\in M$. 

\smallskip
A morphism $g:(M,\righthalfcap_M)\longrightarrow (M',\righthalfcap_{M'})$ of $(B,\righthalfcap_B)$-modules is simply a morphism of modules over the distributive lattice underlying $B$.
\end{defn}

In Section 3, we will give one way of constructing Boolean modules by considering Heyting algebras that satisfy certain conditions. 

\smallskip
Throughout this section and the rest of this paper, we  try to minimize the use of subscripts whenever the meaning is clear from context. As such, we will   write $0$ for the least elements, $1$
for the largest elements and $\righthalfcap$ for  the unary operators  wherever applicable. 

\begin{lem}\label{L2.4} Let $M$ be a Boolean module over a Boolean algebra $B$. Then, the operation $\righthalfcap_M:
M\longrightarrow B$ is order reversing.
\end{lem}

\begin{proof}
We consider $m$, $n\in M$ such that $m\leq n$, i.e., $m\vee n=n$. It follows that:
\begin{equation}
\righthalfcap(n)=\righthalfcap(m\vee n)=\righthalfcap(m)\wedge \righthalfcap(n)\leq \righthalfcap(m)
\end{equation} This proves the result. 
\end{proof} 

The next result shows that when we restrict to finite Boolean algebras, every distributive module is already a Boolean module.  

\begin{thm}\label{P2.5} Suppose that $(B,\righthalfcap_B)$ is a finite Boolean algebra and let $M$ be a distributive module over  $B$. Then, there is a canonical map $\righthalfcap_M:M\longrightarrow B$ making $(M,\righthalfcap_M)$ a Boolean module over $(B,\righthalfcap_B)$. 
\end{thm}

\begin{proof} Since $B$ is finite, for any $m\in M$, we can set $\righthalfcap_M(m)$ to be the finite join:
\begin{equation}
\righthalfcap_M(m):=\bigvee \{\mbox{$x\in B$ $\vert$ $x\wedge m=0$ }\}\in B
\end{equation} Since $\_\_\wedge m: B\longrightarrow M$ preserves finite joins, it is clear that $\righthalfcap(m)\wedge m=0$. 
As such, we see that
\begin{equation}
x\wedge m =0 \qquad\Leftrightarrow \qquad x\leq \righthalfcap(m)
\end{equation} for any $x\in B$. It now follows that for any $y\in B$ we have:
\begin{equation}\label{eq2.5}
y\leq \righthalfcap (x\wedge m)\quad\Leftrightarrow \quad y\wedge x\wedge m=0 \quad\Leftrightarrow \quad y\wedge x\leq \righthalfcap(m)
\end{equation} For any elements $p$, $q$, $r$ in a Boolean algebra, it may be verified easily that $p\wedge q\leq r$ 
is equivalent to $p\leq \righthalfcap(q)\vee r$. Applying this in \eqref{eq2.5} we get
\begin{equation}
y\leq \righthalfcap (x\wedge m) \quad\Leftrightarrow\quad y\leq \righthalfcap(x)\vee \righthalfcap(m)
\end{equation} for any $y\in B$. This gives $ \righthalfcap (x\wedge m) = \righthalfcap(x)\vee \righthalfcap(m)$. On the other hand, for $m$, $n\in M$, we have:
\begin{equation}
\begin{array}{ll}
y\leq \righthalfcap (m\vee n) & \Leftrightarrow\quad y\wedge (m\vee n)=0\\
&\Leftrightarrow\quad \mbox{$(y\wedge m)=0$ and $(y\wedge n=0)$}\\
&\Leftrightarrow\quad \mbox{$(y\leq \righthalfcap(m))$ and $(y\leq \righthalfcap(n))$}\\
&\Leftrightarrow\quad y\leq \righthalfcap(m)\wedge \righthalfcap(n)\\
\end{array}
\end{equation} Hence, we have $ \righthalfcap (m\vee n) =\righthalfcap(m)\wedge \righthalfcap(n)$. This proves the 
result. 

\end{proof}

\section{Heyting Modules}

\smallskip
We begin by recalling the notion of a Heyting algebra, which generalizes the concept of a Boolean algebra. A partially ordered set $(L,\leq )$ may be viewed as a category $L$ whose objects are the elements of $L$ and such that there is a single morphism
$x\longrightarrow y$ in $L$ whenever $x\leq y\in L$. When $L$ is a lattice, for each fixed $y\in L$, the association 
$x\mapsto x\wedge y$, $\forall$ $x\in L$ defines a functor from $L$ to $L$. 

\begin{defn}\label{D3.1} (see \cite[$\S$ I.1.10]{PJT})  A Heyting algebra $H$ is a lattice such that for each fixed $y\in H$, the functor defined by the association $x\mapsto x\wedge y$,  $\forall$ $x\in H$ has a right adjoint. In other words, for any elements $y$, $z\in H$, there is an element $(y\rightarrow z)_H\in H$ such that 
\begin{equation*}
 x\wedge y\leq z\quad\Leftrightarrow\quad  x\leq (y\rightarrow z)_H
\end{equation*} for any $x\in H$. 
\end{defn}

In general, every Heyting algebra is distributive. For $y\in H$, the element $\righthalfcap_H(y):=(y\rightarrow 0)_H\in H$ is referred to as the negation of $y$. If the operation $\righthalfcap_H$ satisfies
$\righthalfcap_H \righthalfcap_H(y)=y$ for every $y\in H$, then $H$ becomes a Boolean algebra (see \cite[Lemma I.1.11]{PJT}). 

\smallskip
We are now ready to introduce the concept of a Heyting module. 

\begin{defn}\label{DD3.2}
Let $H$ be a Heyting algebra. A Heyting module over $H$ consists of the following data:

\smallskip
(1) A distributive module $M$ over the distributive lattice underlying $H$. 

\smallskip
(2) For each $m\in M$, the functor $H\longrightarrow M$ defined by the association $x\mapsto x\wedge m$ has a right adjoint. In other words, for any $m$, $n\in M$, there is an element $(m\rightarrow n)_M\in H$ such that 
\begin{equation*}
x\wedge m\leq n\quad\Leftrightarrow \quad x\leq (m\rightarrow n)_M
\end{equation*} for any $x\in H$. 

\smallskip
A morphism of Heyting modules is simply a morphism of modules over the distributive lattice underlying $H$. We denote by $Heymod_H$ the category of Heyting modules over a Heyting algebra $H$. 
\end{defn}

Given a Heyting module $M$ and an element $m\in M$, we define the negation $\righthalfcap_M(m):=(m\rightarrow 0)_M$. Again, we will generally omit the subscripts whenever the meaning is clear from context. 

\begin{thm}\label{P3.25}
Let $f:M\longrightarrow N$ be a morphism of Heyting modules over a Heyting algebra $H$. Suppose that $f:M\longrightarrow N$
has a right adjoint $g:N\longrightarrow M$. Then, for any $m\in M$ and $n\in N$, we have $(f(m)\rightarrow n)_{N}=(m
\rightarrow g(n))_M$ in $H$.
\end{thm}

\begin{proof}
Since $f:M\longrightarrow N$ is a morphism of  Heyting modules, we know that for any $m\in M$, the functor $\_\_\wedge f(m):H
\longrightarrow N$ can be expressed as the composition
\begin{equation}
\_\_\wedge f(m) = f\circ (\_\_\wedge m):H\longrightarrow N
\end{equation} As such, the right adjoint of $\_\_\wedge f(m)$ is equal to the right adjoint $g:N\longrightarrow M$
composed with the right adjoint $(m\rightarrow \_\_)_M:M\longrightarrow H$. 
\end{proof}

\begin{lem}\label{L3.3} Let $M$ be a Heyting module over a Heyting algebra $H$. Then:

\smallskip
(a) For $m$, $m'$, $n\in M$ we have
\begin{equation}
(m\rightarrow n)\wedge (m'\rightarrow n)=((m\vee m')\rightarrow n)\qquad 
\end{equation} In particular, we know that $\righthalfcap(m\vee m')=\righthalfcap(m)\wedge \righthalfcap(m')$. 

\smallskip
(b) For $x\in H$ and $m$, $n\in M$ we have
\begin{equation}
((x\wedge m)\rightarrow n) = (x\rightarrow (m\rightarrow n))
\end{equation} In particular, we know that $\righthalfcap(x\wedge m)=(x\rightarrow \righthalfcap(m))$. 

\smallskip
(c) For each $n\in M$, the association  $(\_\_\rightarrow n): M\longrightarrow H$ is order reversing. 

\smallskip
(d) $\righthalfcap_M(0_M)=1_H$. 
\end{lem}

\begin{proof} (a) For any $x\in H$, we see that
\begin{equation}
\begin{array}{ll}
x\leq ((m\vee m')\rightarrow n)& \Leftrightarrow\quad (x\wedge m)\vee (x\wedge m')\leq n \\
&\Leftrightarrow\quad  (x\wedge m)\leq n \quad \&\quad (x\wedge m')\leq n\\
&\Leftrightarrow\quad  (x\leq (m\rightarrow n))\quad\&\quad (x\leq (m'\rightarrow n))\\
&\Leftrightarrow\quad  x\leq (m\rightarrow n)\wedge (m'\rightarrow n)\\
\end{array}
\end{equation}

\smallskip
(b) For any $y\in H$, we see that
\begin{equation}
y\leq ((x\wedge m)\rightarrow n) \quad\Leftrightarrow\quad  (y\wedge x\wedge m)\leq n  \quad\Leftrightarrow\quad
y\wedge x\leq (m\rightarrow n) \quad\Leftrightarrow\quad y\leq (x\rightarrow (m\rightarrow n))
\end{equation} The result of (c) is an immediate consequence of (a). The result of (d) is clear from the fact that $1_H\wedge 0_M=0_M$. 

\end{proof}

\begin{thm}\label{P3.4} Let $M$ be a Heyting module over a Heyting algebra $H$. Then, if $(H,\righthalfcap_H)$ is a Boolean algebra, $(M,\righthalfcap_M)$ is a Boolean module over $H$. 
\end{thm}
\begin{proof}
Since the Heyting algebra $H$ is actually Boolean, we know (see \cite[$\S$ I.1.10]{PJT}) that $(x\rightarrow y)_H=\righthalfcap_H(x)\vee y$ for any elements $x$, $y\in H$. From Lemma \ref{L3.3}(b) it now follows that for $x\in H$, $m\in M$, we have:
\begin{equation}
\righthalfcap_M(x\wedge m)=(x\wedge m\rightarrow 0)=(x\rightarrow \righthalfcap_M(m))=\righthalfcap_H(x)\vee \righthalfcap_M(m)
\end{equation} The other conditions mentioned in \eqref{eq2.2} that $(M,\righthalfcap_M)$ must satisfy in order to be a Boolean module are also clear from Lemma \ref{L3.3}. 
\end{proof}

We will now show that when $H$ is a finite Heyting algebra, every distributive module is already a Heyting module. 

\begin{thm}\label{P3.5}  Suppose that $H$ is a finite Heyting algebra and let $M$ be a distributive module over  $H$. Then, there is a canonical map $(\_\_\rightarrow\_\_)_M:M\times M\longrightarrow H$ making $M$ a Heyting module over $H$. 
\end{thm}

\begin{proof} We fix some $m\in M$. Then, the  right adjoint of the functor 
$\_\_\wedge m : H\longrightarrow M$ is given by setting
\begin{equation}\label{e3.7pg}
(m\rightarrow n)_M:=\bigvee\{\mbox{$x\in H$ $\vert$ $x\wedge m\leq n$}\}
\end{equation} for each $n\in M$. This gives $M$ the canonical structure of a Heyting module.

\end{proof}

\begin{cor}\label{C3.6ty} Suppose that $H$ is a finite Heyting algebra and $g:M\longrightarrow N$ is a morphism in $Heymod_H$. Then, for any
$m$, $m'\in M$, we have $(m\rightarrow m')_M\leq (g(m)\rightarrow g(m'))_{M'}\in H$.
\end{cor}

\begin{proof} From \eqref{e3.7pg}, we know that $(m\rightarrow m')_M$ is the join of all the elements $x\in H$ such that
$x\wedge m\leq m'$. Since $g$ is a morphism in $Heymod_H$, $x\wedge m\leq m'$ $\Rightarrow$ $x\wedge g(m)=g(x\wedge m)\leq g(m')$. 
The result is now clear.  

\end{proof}

\begin{rem}\label{XR3.7}\emph{(1) The proof of Proposition \ref{P3.5} is a special case of the more general Adjoint Functor Theorem 
for join semilattices (see, for instance, \cite[$\S$ I.4.2]{PJT}).}

\smallskip
\emph{(2) It is known (see, for instance, \cite[Exercise I.1.12(ii)]{PJT}) that every finite distributive lattice is a Heyting
algebra. Conversely, we know that every Heyting algebra is distributive and a lattice. As such, a finite Heyting algebra is simply
a finite distributive lattice.}

\smallskip
\emph{(3) It is easy to give examples of finite Heyting algebras which are not finite Boolean algebras. For instance, we may consider any finite totally
ordered set $H=\{0=x_1<x_2<x_3 .... <x_k=1 \}$. Then, $H$ becomes a Heyting algebra (see \cite[$\S$ I.1.12]{PJT}) by setting}
\begin{equation}
(x_m\rightarrow x_n)_H:=\left\{\begin{array}{ll} 1 & \mbox{if $m\leq n$} \\ x_n & \mbox{otherwise}\\
\end{array}\right.
\end{equation}
\emph{Then, we see that every $x_m\ne 0$ satisfies $\righthalfcap\righthalfcap x_m=1$. As such, $H$ cannot be a Boolean algebra for any $k\geq 3$. }

\end{rem}

We record here
the following fact. 

\begin{thm}\label{P4.5} (a) In the category $Heymod_H$ of Heyting modules over a Heyting algebra $H$, a morphism $g:M\longrightarrow N$ is a monomorphism $\Leftrightarrow$ the underlying map is injective. 

\smallskip
(b) If $g:M\longrightarrow N$ is a monomorphism in $Heymod_H$, then  for any
$m$, $m'\in M$, we have $(m\rightarrow m')_M= (g(m)\rightarrow g(m'))_{N}\in H$. 
\end{thm}

\begin{proof} (a) The $\Leftarrow$ implication is obvious. On the other hand, for any element $m\in M$, we can consider
the morphism $g_m:H\longrightarrow M$ in $Heymod_H$ taking each $c\in H$ to $c\wedge m\in M$. Then, if $m$, $m'\in M$ are
such that $g(m)=g(m')\in N$, then $g\circ g_m=g\circ g_{m'}$. If $g$ is a monomorphism, this implies $g_m=g_{m'}$ and hence
$m=g_m(1_H)=g_{m'}(1_H)=m'$. 

\smallskip
(b) From part (a), it follows that $M\subseteq N$. As such, for $m$, $m'\in M$, we have $x\wedge m\leq m'$ if and only if 
$x\wedge g(m)\leq g(m')$. The result now follows from the definitions. 
\end{proof}

\begin{thm}\label{Pp3.65} Let $M$ be a Heyting module over a Heyting algebra $H$. Let $N\subseteq M$ be a distributive submodule of
$M$ over the lattice $H$. Then, $N$ is a Heyting module over $H$ and $(n\rightarrow n')_N=(n\rightarrow n')_M$ for $n$, $n'\in N$. 
\end{thm}

\begin{proof}
We set $(n\rightarrow n')_N:=(n\rightarrow n')_M$ for $n$, $n'\in N$. Since the partial ordering on $N\subseteq M$ is induced by the partial 
ordering on $M$, we get $c\wedge n\leq n'$ $\Leftrightarrow$ $c\leq (n\rightarrow n')_M=(n\rightarrow n')_N$. This proves the result. 
\end{proof}

\begin{defn}\label{D3.7} (see \cite[$\S$ 1.13]{PJT}) An element $x$ in a Heyting algebra $H$ is said to be regular if it satisfies
$\righthalfcap \righthalfcap (x)=x$. The collection  of regular elements of $H$ is denoted by $H_{\righthalfcap\righthalfcap}$.

\end{defn}

The regular elements $H_{\righthalfcap\righthalfcap}$ of a Heyting algebra $H$ form a Boolean algebra. The complementation 
in $H_{\righthalfcap\righthalfcap}$ coincides with  the negation on $H$. While meets in $H_{\righthalfcap\righthalfcap}$  always coincide with meets in  $H$, the same is not necessarily true for joins in $H_{\righthalfcap\righthalfcap}$. In fact, $H_{\righthalfcap\righthalfcap}$ is a sublattice of $H$ if and only if the negation operator
 satisfies $\righthalfcap (x\wedge y)=\righthalfcap(x)\vee \righthalfcap(y)$ for every $x$, $y\in H$. We recall that the dual  
 of this relation, i.e., 
$\righthalfcap (x\vee y)=\righthalfcap(x)\wedge \righthalfcap(y)$ for $x$, $y\in H$  holds in every Heyting algebra. 

\smallskip
We now suppose that $H_{\righthalfcap\righthalfcap}$ is a sublattice of $H$. Given a Heyting module $M$ over $H$, we consider
\begin{equation}\label{eq3.8}
M_{\righthalfcap\righthalfcap}:=\{\mbox{$m\in M$ $\vert$ $\righthalfcap(m)\in H$ is regular, i.e., $\righthalfcap(m)\in 
H_{\righthalfcap\righthalfcap}$}\}
\end{equation}
We will say that the elements of $M_{\righthalfcap\righthalfcap}$ are the regular elements of $M$ over $H$. We will show that $M_{\righthalfcap\righthalfcap}$ is actually a Boolean module over $H_{\righthalfcap\righthalfcap}$. For this, we need the following simple result.

\begin{lem}\label{L3.9} Let $H$ be a Heyting algebra such that the Boolean algebra $H_{\righthalfcap\righthalfcap}$ is a sublattice of $H$. Then, given any regular elements $a$, $b\in H_{\righthalfcap\righthalfcap}$, we have $(a\rightarrow b)_H=\righthalfcap(a)\vee b$. In particular, $(a\rightarrow b)_H$ is regular. 
\end{lem}

\begin{proof} For any elements $a$, $b$  in a Heyting algebra $H$, we know (see, for instance, \cite{Hey}) that
$\righthalfcap (a\rightarrow b)=(\righthalfcap\righthalfcap(a))\wedge (\righthalfcap(b))$. If $a$ and $b$ are regular, this implies
that 
\begin{equation}
\righthalfcap (a\rightarrow b)=a\wedge \righthalfcap(b)
\end{equation} Since  $H_{\righthalfcap\righthalfcap}$ is a sublattice of $H$, it follows that
 $\righthalfcap\righthalfcap(a\rightarrow b)=\righthalfcap(a\wedge \righthalfcap(b))=\righthalfcap(a)\vee \righthalfcap\righthalfcap(b)=\righthalfcap(a)\vee b$. For any element $x$ in a Heyting algebra, it is clear that $x\leq \righthalfcap\righthalfcap(x)$ (since $\righthalfcap(x)\leq \righthalfcap(x)$ which gives $x\wedge\righthalfcap(x) =0$). It follows that
 \begin{equation}\label{eq3.10}
 \righthalfcap\righthalfcap(a\rightarrow b)\leq \righthalfcap(a)\vee b\quad \Rightarrow \quad(a\rightarrow b)\leq \righthalfcap(a)\vee b
 \end{equation} Conversely, we have
 \begin{equation}\label{eq3.11}
 a\wedge b\leq b \quad\Rightarrow\quad a\wedge (\righthalfcap(a)\vee b)\leq b \quad\Rightarrow\quad (\righthalfcap(a)\vee b)\leq (a\rightarrow b)
 \end{equation} From \eqref{eq3.10} and \eqref{eq3.11} the result follows. 
\end{proof}

\begin{thm}\label{P3.10}  Let $H$ be a Heyting algebra such that the Boolean algebra $H_{\righthalfcap\righthalfcap}$ is a sublattice of $H$. Suppose that $M$ is a Heyting module over $H$. Then, $M_{\righthalfcap\righthalfcap}$ is a Boolean module
over the Boolean algebra $H_{\righthalfcap\righthalfcap}$.
\end{thm}

\begin{proof}
If $m$, $n\in M_{\righthalfcap\righthalfcap}$, we know from Lemma \ref{L3.3}(a) that $\righthalfcap(m\vee n)=\righthalfcap(m)\wedge \righthalfcap(n)$. It follows that $m\vee n\in M_{\righthalfcap\righthalfcap}$. For $x\in H_{\righthalfcap\righthalfcap}$ and $m\in M_{\righthalfcap
\righthalfcap}$, it follows from Lemma \ref{L3.3}(b) that $\righthalfcap(x\wedge m)=(x\rightarrow \righthalfcap(m))$. Since 
$x\in H_{\righthalfcap\righthalfcap}$ and $m\in M_{\righthalfcap\righthalfcap}$, we can apply Lemma \ref{L3.9} to see that
\begin{equation}\righthalfcap(x\wedge m)=(x\rightarrow \righthalfcap(m))=\righthalfcap(x)\vee \righthalfcap(m) \in H_{\righthalfcap\righthalfcap}
\end{equation}  Hence, $x\wedge m\in M_{\righthalfcap\righthalfcap}$.  The map $M_{\righthalfcap\righthalfcap}
\longrightarrow H_{\righthalfcap\righthalfcap}$ is obtained by restricting the map $\righthalfcap : M\longrightarrow H$. From the above, it is clear that this map satisfies the conditions for making $M$ into a Boolean module over the Boolean algebra $H_{\righthalfcap\righthalfcap}$. 
\end{proof} 

An application of Proposition \ref{P3.10} gives us a simple way of constructing examples of Boolean modules.

\begin{cor}\label{C3.11}
Let $H$ be a Heyting algebra such that the Boolean algebra $H_{\righthalfcap\righthalfcap}$ is a sublattice of $H$.  Then, 
$H$ is canonically equipped with the structure of  a Boolean module over the Boolean algebra $H_{\righthalfcap\righthalfcap}$. 
\end{cor}

\begin{proof} It is clear that $H$ is a Heyting module over itself. To prove the result, it suffices to show that every element 
$x\in H$ is regular over $H$ in the sense of 
\eqref{eq3.8}, i.e., $\righthalfcap(x)\in H_{\righthalfcap\righthalfcap}$ for any $x\in H$. For this, we notice that $x\leq \righthalfcap\righthalfcap(x)$, $\forall$ $x\in H$ as mentioned in the proof of Lemma \ref{L3.9}.  This gives
$\righthalfcap(x)\leq \righthalfcap\righthalfcap\righthalfcap(x)$. However, since $\righthalfcap$ is order reversing, the relation
$x\leq \righthalfcap\righthalfcap(x)$ gives $\righthalfcap(x)\geq \righthalfcap\righthalfcap\righthalfcap(x)$. Hence, $\righthalfcap(x)=\righthalfcap\righthalfcap\righthalfcap(x)$ for every $x\in H$, i.e., $\righthalfcap(x)\in H_{\righthalfcap\righthalfcap}$.
\end{proof}

\section{Duals of Heyting modules and hereditary systems}

\smallskip
We continue with $H$ being a Heyting algebra and $M$ being a Heyting module over it. By a Heyting submodule of $M$, we will
mean a distributive submodule $N$ of $M$ over the lattice underlying $H$. From Proposition \ref{Pp3.65}, it is clear that $N$ is canonically equipped with the structure of a Heyting module, with $(n_1\rightarrow n_2)_{N}=(n_1\rightarrow n_2)_M$ for any $n_1$, $n_2\in N$. 

\smallskip
The following simple observation will be very useful to us: by definition, the element $1_H$ in the lattice $H$ is the largest
element in $H$, i.e., every $c\in H$ satisfies $c\leq 1$. Hence, for any $m\in M$, we have
\begin{equation}\label{e4.0}
c\wedge m \leq 1_H\wedge m = m
\end{equation}

\smallskip
For a Heyting module $M$, we will now consider two different  ``dual objects'' : the first is the collection $Heymod_H(M,H)$ of Heyting module
morphisms from $M$ to $H$ which we will denote by $M^\bigstar$. The second dual object, which we shall denote by
$M^\star$, is the  collection of join semilattice homomorphisms $\phi:M\longrightarrow 
\{0,1\}$. Since any $\phi\in M^\star$ is order preserving, we notice that 
\begin{equation}\label{e4.1}
\phi(m)=0 \quad\Rightarrow\quad \phi(c\wedge m)\leq \phi(m)=0 \quad\Rightarrow \quad\phi(c\wedge m)=0 \textrm{ }\forall\textrm{ }c\in H
\end{equation} It is clear that a morphism $M_1\longrightarrow M_2$ of Heyting modules induces  maps
$M_2^\bigstar\longrightarrow M_1^\bigstar$ and $M_2^\star\longrightarrow M_1^\star$ between their respective duals.

\begin{defn}\label{D4.1} Let $M$ be a Heyting module over a Heyting algebra $H$. A subjoin semilattice $N\subseteq M$ will be called a hereditary Heyting submodule if it satisfies the following condition
\begin{equation*}
n\in N \quad\&\quad n'\leq n \quad\Rightarrow n'\in N
\end{equation*}
\end{defn} 
From the observation in \eqref{e4.0}, it is immediate that any hereditary Heyting submodule of $M$ is automatically a Heyting submodule. As such, a hereditary Heyting submodule of $M\in Heymod_H$ in the sense of Definition \ref{D4.1} is simply a  hereditary submodule of the underlying join semilattice $M$ 
 in the sense of \cite[Definition 2.2]{one}. 
 
 \smallskip 
For each element $m\in M$, we now consider:
\begin{equation}\label{eq4.1}
I_m^M=I_m:=\{\mbox{$n\in M$ $\vert$ $n\leq  m$ }\}
\end{equation} It is easily seen that each $I_m$ is a hereditary Heyting submodule. 

\begin{thm}\label{P4.2} Let $M$ be a Heyting module over a Heyting algebra $H$. Then,

\smallskip
(a) There is a one-one correspondence between elements of $M^\star$ and hereditary Heyting submodules of $M$. 

\smallskip
(b) A distributive submodule of $M$ over $H$ is hereditary if and only if it is a filtering union of hereditary Heyting submodules
of the form $\{I_m\}_{m\in M}$. 

\smallskip
(c) If $m$, $m'\in M$ are such that $\phi(m)=\phi(m')$ for every $\phi\in M^\star$, then $m=m'$. 
\end{thm}

\begin{proof} As mentioned before, $N\subseteq M$ is a hereditary Heyting submodule if and only if it is a hereditary submodule
of the join semilattice $M$ in the sense of \cite[Definition 2.2]{one}. As such, all three  parts (a), (b) and (c) follow directly
from \cite[Proposition 2.3]{one}. Explicitly, an element $\phi\in M^\star$ corresponds to the hereditary submodule 
$\phi^{-1}(0)\subseteq M$. 
\end{proof}

We now introduce the notion of a ``hereditary system of submodules'' which will be used to describe elements of the dual
$M^\bigstar$. First of all, for a given hereditary submodule $K\subseteq M$ we set 
\begin{equation}
(K:c):=\{\mbox{$m\in M$  $\vert$ $c\wedge m\in K$}\}
\end{equation} for each $c\in K$. It is clear that $(K:c)\subseteq M$ is also a hereditary Heyting submodule.

\begin{defn}\label{D4.3} Let $M$ be a Heyting module over a Heyting algebra $H$. A hereditary system $\mathcal K=\{K_c\}_{c\in H}$ of submodules of $M$ consists of the following data:

\smallskip
(a) For each $c\in H$, $K_c$ is a hereditary submodule of the Heyting module $M$.

\smallskip
(b) Given any elements $c$, $d\in H$, we have $K_c\cap K_d=K_{c\wedge d}$.

\smallskip
(c) Given any elements $c$, $d\in H$, we have $(K_d:c)=K_{(c\rightarrow d)}$. 

\smallskip In particular, if 
$c\leq d$ in $H$, we have $K_c\subseteq K_d$. 
\end{defn}

\begin{lem}\label{L4.31}
Let $M$ be a Heyting module and let $\phi:M\longrightarrow H$ be an element of $M^\bigstar$. Then, 
setting $K_c:=\{\mbox{$m\in M$ $\vert$ $\phi(m)\leq c$}\}$ for each $c\in H$ gives a hereditary system $\mathcal K_\phi$ on $M$. 
\end{lem}

\begin{proof}
Since $\phi\in M^\bigstar$ is order preserving and preserves joins, we see that each $K_c=\{\mbox{$m\in M$ $\vert$ $\phi(m)\leq c$}\}$ is hereditary. Since $\phi:M\longrightarrow H$ is $H$-linear, for $c$, $d\in H$ and any $m\in M$, we have:
\begin{equation*}
\begin{array}{c}
m\in K_c\cap K_d  \textrm{ }\Leftrightarrow\textrm{ }  \phi(m)\leq c \textrm{ }\&\textrm{ }\phi(m)\leq d \textrm{ }\Leftrightarrow\textrm{ } \phi(m)\leq c\wedge d \textrm{ }\Leftrightarrow\textrm{ }  m\in K_{c\wedge d}\\
m\in (K_d:c) \textrm{ }\Leftrightarrow\textrm{ } c\wedge m\in K_d \textrm{ }\Leftrightarrow\textrm{ } \phi(c\wedge m)\leq d \textrm{ }\Leftrightarrow\textrm{ } c\wedge \phi(m)\leq d
\textrm{ }\Leftrightarrow\textrm{ } \phi(m)\leq (c\rightarrow d)\textrm{ }\Leftrightarrow\textrm{ }m\in K_{(c\rightarrow d)}\\
\end{array}
\end{equation*} This proves the result. 
\end{proof}

\begin{thm}\label{P4.32}
Let $H$ be a finite Heyting algebra and let $M$ be a Heyting module over $H$. Then, there is a one-one correspondence between
$M^\bigstar$ and hereditary systems of submodules of $M$. 
\end{thm}
\begin{proof}
Given $\phi\in M^\bigstar$, we have already constructed a hereditary system $\mathcal K_\phi$ of submodules of $M$ in Lemma \ref{L4.31}. Conversely, given a hereditary system $\mathcal K=\{K_c\}_{c\in H}$, we define $\phi_{\mathcal K}:M
\longrightarrow H$ by setting
\begin{equation}\label{eq4.5r}
\phi_{\mathcal K}(m):=\underset{m\in K_d}{\bigwedge} d
\end{equation} We begin by showing that $\phi_{\mathcal K}:M\longrightarrow H$ preserves joins. We pick $m_1$, $m_2\in M$ and 
set $d_1=\phi_{\mathcal K}(m_1)$, $d_2=\phi_{\mathcal K}(m_2)$. Since $H$ is finite, it follows from condition (b) in Definition \ref{D4.3} that $d_1$ (resp. $d_2$) 
is the smallest element of $H$ such that $m_1\in K_{d_1}$ 
(resp. $m_2\in K_{d_2}$).  

\smallskip
Since $K_{d_1}$, $K_{d_2}\subseteq K_{d_1\vee d_2}$, we have $m_1\vee m_2\in K_{d_1\vee d_2}$. On the other hand, 
if $e\in H$ is such that $m_1\vee m_2\in K_e$, then $m_1,m_2\leq m_1\vee m_2$ $\Rightarrow$ $m_1$, $m_2\in K_e$ (since
$K_e$ is hereditary). Then, $e\geq d_1$ and $e\geq d_2$ which gives $e\geq d_1\vee d_2$. From \eqref{eq4.5r}, it now follows
that $\phi_{\mathcal K}(m_1\vee m_2)=d_1\vee d_2=\phi_{\mathcal K}(m_1)\vee \phi_{\mathcal K}(m_2)$. 

\smallskip
For $c\in H$ and $m\in M$, we know that
\begin{equation}\label{eq4.6r}
c\wedge \phi_{\mathcal K}(m)=c\wedge \left(\underset{m\in K_d}{\bigwedge} d\right)=\underset{m\in K_d}{\bigwedge}(c\wedge d)\qquad \phi_{\mathcal K}(c\wedge m)=\underset{c\wedge m\in K_d}{\bigwedge} d = \underset{m\in K_{(c\rightarrow d)}}{\bigwedge} d
\end{equation} Since $d\leq (c\rightarrow c\wedge d)$,  we notice that if $m\in K_d$, then $m\in K_{ (c\rightarrow c\wedge d)}=
(K_{c\wedge d}:c)$. Hence $c\wedge m\in K_{c\wedge d}$. From \eqref{eq4.6r}, it now follows that
$c\wedge \phi_{\mathcal K}(m)\geq \phi_{\mathcal K}(c\wedge m)$. Conversely, for any $d\in H$ such that $m\in K_{(c\rightarrow d)}$ it follows from the definition in \eqref{eq4.5r}
that $\phi_{\mathcal K}(m)\leq (c\rightarrow d)$. Then
\begin{equation}\label{eq4.7r}
\phi_{\mathcal K}(m)\leq \underset{m\in K_{(c\rightarrow d)}}{\bigwedge} (c\rightarrow d)=\left(c\rightarrow \left(\underset{m\in K_{(c\rightarrow d)}}{\bigwedge} d\right)\right)=(c\rightarrow \phi_{\mathcal K}(c\wedge m))\quad \Rightarrow \quad 
c\wedge \phi_{\mathcal K}(m)\leq \phi_{\mathcal K}(c\wedge m)
\end{equation} Hence, we have $c\wedge \phi_{\mathcal K}(m)=\phi_{\mathcal K}(c\wedge m)$ and it follows that
$\phi_{\mathcal K}\in M^\bigstar$. 

\smallskip
It remains to show that the two associations are inverse to each other. First of all, it is clear that  $\phi=\phi_{\mathcal K_\phi}$
for any $\phi\in M^\bigstar$. On the other hand, it follows from the above that $m\in K_{\phi_{\mathcal K}(m)}$ for each
$m\in M$ and hence $\mathcal K_{\phi_{\mathcal K}}=\mathcal K$.  
\end{proof}

The next result shows that when the Heyting algebra $H$ is actually Boolean, the hereditary systems take a particularly simple form.

\begin{thm}\label{P4.33} Let $M$ be a Heyting module over a Heyting algebra $H$. If $H$ is a Boolean algebra, then 
there is a one-one correspondence between hereditary systems of submodules of $M$ and hereditary submodules of $M$.
\end{thm}

\begin{proof}
Let $\mathcal K=\{K_c\}_{c\in H}$ be a hereditary system of submodules of $M$. If $H$ is Boolean, we claim that $\mathcal K$
is determined completely by the hereditary submodule $K_0$. This is because condition (c) in Definition \ref{D4.3} reduces to 
\begin{equation}
K_c=K_{(\righthalfcap c\rightarrow 0)}=(K_0:\righthalfcap c)
\end{equation} To complete the proof, it remains to show that for any hereditary submodule $K\subseteq M$, 
the collection $K_c:=(K:\righthalfcap c)$, $c\in H$ gives a hereditary system of submodules of 
$M$. We have noted before that each $(K:\righthalfcap c)$ is hereditary. Further, for any $m\in M$ and $c$, $d\in H$, we have:
\begin{equation*}
m\in (K_d:c) \textrm{ }\Leftrightarrow \textrm{ }c\wedge m\in K_d=(K:\righthalfcap d)  \textrm{ }\Leftrightarrow \textrm{ }m\in  (K:c\wedge \righthalfcap d)=(K:\righthalfcap (\righthalfcap c \vee d))=(K:\righthalfcap (c\rightarrow d))=K_{(c\rightarrow d)}
\end{equation*} Here we have used the fact that since $H$ is Boolean, we must have $(c\rightarrow d)=\righthalfcap c\vee d$. Finally, since $K$ is hereditary, we know that for $c$, $d\in H$ and $m\in M$, $(\righthalfcap c \vee \righthalfcap d)\wedge m
\in K$ if and only if both $\righthalfcap c\wedge m$, $\righthalfcap d\wedge m\in K$. It follows that
\begin{equation*}
m\in K_c\cap K_d \textrm{ }\Leftrightarrow \textrm{ } \righthalfcap c\wedge m\textrm{ }\&\textrm{ }\righthalfcap d\wedge m\in K
\textrm{ }\Leftrightarrow \textrm{ } (\righthalfcap c \vee \righthalfcap d)\wedge m
\in K \textrm{ }\Leftrightarrow \textrm{ }m\in (K:\righthalfcap (c\wedge d))=K_{c\wedge d}
\end{equation*} This proves the result. 
\end{proof}

\begin{thm}\label{xP4.34} Let $M$ be a Heyting module over $H$. If $H$ is a finite Boolean algebra, then there is a one-one correspondence between the duals $M^\bigstar$ and $M^*$. 
\end{thm}

\begin{proof}
This result follows directly as a consequence of Propositions \ref{P4.2}, \ref{P4.32} and \ref{P4.33}. This correspondence may
be made explicit as follows : given $\phi:M\longrightarrow H$ in $M^\bigstar$, we can define $\psi:M\longrightarrow \{0,1\}$
in $M^*$ by setting $\psi(m)=0$ if $\phi(m)=0$ and $\psi(m)=1$ otherwise. 

\smallskip
Conversely, given $\psi:M\longrightarrow \{0,1\}$
in $M^*$ , we can obtain a corresponding  $\phi:M\longrightarrow H$ in $M^\bigstar$ by setting
\begin{equation}\label{x4.9e}
\phi(m)=\bigwedge \{\mbox{$c\in H$ $\vert$ $\psi(\righthalfcap c\wedge m)=0$}\}
\end{equation}
\end{proof}

\begin{thm}\label{P4.3}
Let $H$ be a finite Boolean algebra, $M$ be a Heyting module over $H$ and let $N\subseteq M$ be a Heyting submodule. Then, the induced map $M^\bigstar\longrightarrow N^\bigstar$ is surjective. 
\end{thm}

\begin{proof}
From Proposition \ref{xP4.34}, we know that there are bijections $M^\bigstar\simeq M^*$ and $N^\bigstar\simeq N^*$. From 
\cite[Proposition 2.3]{one}, we know that the induced morphism $M^*\longrightarrow N^*$ on the duals of the join semilattices
is surjective. The result is now clear. 
\end{proof}

\begin{thm}\label{P4.31}
Let $H$ be a finite Boolean algebra and let $M$ be a Heyting module over $H$. Then, the duality between $M$ and $M^\bigstar$ is separating, i.e., 
if $m_1$, $m_2\in M$ are such that $\phi(m_1)=\phi(m_2)$ for every $\phi\in M^\bigstar$, then $m_1=m_2$. 
\end{thm}

\begin{proof}
We consider the hereditary submodule $I_{m_1}\subseteq M$ as defined in \eqref{eq4.1}. Then, $\{(I_{m_1}:\righthalfcap c)\}_{c\in H}$
is a hereditary system of submodules of $M$ and let $\phi$ be the corresponding element in $M^\bigstar$. Then, $\phi(m)=0$
for some $m\in M$ if and only if $m\in I_{m_1}$. Then, $\phi(m_2)=\phi(m_1)=0$ and hence $m_2\in I_{m_1}$, i.e., 
$m_2\leq m_1$. Similarly, we can show that $m_1\leq m_2$. Hence, $m_1=m_2$. 
\end{proof}

 We now obtain the following
version of Hahn-Banach Theorem for modules over a finite Boolean algebra  (compare \cite[Lemma 2.6]{one} and also \cite{CGQ}). 

\begin{thm}\label{P4.4}
Let $H$ be a finite Boolean algebra, $M$ be a Heyting module over $H$ and $i:N\hookrightarrow M$ be a Heyting submodule. Then, if $m\in M$ is such that
$\phi_1(m)=\phi_2(m)$ for every pair  $(\phi_1,\phi_2)\in M^\bigstar\times M^\bigstar$ such that $\phi_1\circ i=\phi_2\circ i$, then $m\in N$.
\end{thm}

\begin{proof}
We consider a pair $(\psi_1,\psi_2)\in M^*\times M^*$ such that $\psi_1\circ i=\psi_2\circ i$. Using the bijection $M^\bigstar
\simeq M^*$, we take $\phi_1$, $\phi_2\in M^\bigstar$ corresponding respectively to $\psi_1$, $\psi_2\in M^*$. For any fixed
$n\in N$ and any $c\in H$, we know that $\righthalfcap c\wedge n\in N$. Since $\psi_1\circ i=\psi_2\circ i$, it follows that
\begin{equation}\label{x4.10e}
\psi_1(\righthalfcap c\wedge n)=0\quad\Leftrightarrow\quad \psi_2(\righthalfcap c\wedge n)=0
\end{equation} From \eqref{x4.9e} and \eqref{x4.10e}, it follows that $\phi_1(n)=\phi_2(n)$ for every $n\in N$, i.e.,
$\phi_1\circ i=\phi_2\circ i$. By assumption, we must therefore have $\phi_1(m)=\phi_2(m)$. Again, from the proof of Proposition 
\ref{xP4.34}, we obtain
\begin{equation}
\psi_1(m)=0 \quad\Leftrightarrow\quad \phi_1(m)=0 \quad\Leftrightarrow\quad \phi_2(m)=0 \quad\Leftrightarrow\quad  \psi_2(m)=0
\end{equation} Hence, any maps $\psi_1,\psi_2:M\longrightarrow \{0,1\}$ in $M^*$ such that
$\psi_1\circ i=\psi_2\circ i$ must satisfy $\psi_1(m)=\psi_2(m)$. Applying \cite[Lemma 2.6]{one}, we get $m\in N$. 
\end{proof}

In Proposition \ref{P4.5}, we showed that a morphism $f:M\longrightarrow N$ in $Heymod_H$ is a monomorphism if and only
if it is injective on underlying sets. For epimorphisms, we have a somewhat less general result. 

\begin{thm}\label{P4.51} Let $H$ be a finite Boolean algebra and let $f:M\longrightarrow N$ be a morphism of Heyting modules
over $H$. Then, $f$ is an epimorphism if and only if $f$ is surjective on underlying sets. 
\end{thm}

\begin{proof}
The ``if part'' is obvious. On the other hand, consider a morphism $f:M\longrightarrow N$ in $Heymod_H$. It is clear that the image of
$f$ (as a subset of $N$) is already a Heyting submodule, which we will denote by $E$. 

\smallskip
Suppose that $f$ is an epimorphism in $Heymod_H$. We now
consider a pair $(\phi_1,\phi_2)\in N^\bigstar\times N^\bigstar$ such that $\phi_1|_E=\phi_2|_E$. Then, $\phi_1\circ f=\phi_2\circ f$. Since $\phi_1,\phi_2:N\longrightarrow H$ are morphisms of Heyting modules and $f$ is an epimorphism, we obtain
$\phi_1=\phi_2$. In particular, $\phi_1(n)=\phi_2(n)$ for every $n\in N$. Using Proposition \ref{P4.4}, we get
$n\in E$ for every $n\in N$, i.e., $f$ is surjective.  
\end{proof}

We conclude this section with the following result.

\begin{thm}\label{P4.61}
Let $M$ be a Heyting module over a finite Heyting algebra $H$. Then, the dual $M^\bigstar$ is a Heyting module over $H$ and 
there is a canonical morphism of Heyting modules $M\longrightarrow M^{\bigstar\bigstar}$. 
\end{thm}

\begin{proof} For $c\in H$ and $\phi\in M^\bigstar$, we set $(c\wedge \phi)(m):=\phi(c\wedge m)$ for any $m\in M$. It may be easily verified that this makes $M^\bigstar$ into a distributive module over the lattice underlying $H$. Similarly, we see that $M^{\bigstar\bigstar}$
is a distributive module over $H$ and there is a canonical map $M\longrightarrow M^{\bigstar\bigstar}$ of distributive modules given by
the association $m\mapsto \langle \_\_,m\rangle : M^\bigstar\longrightarrow H$ for each $m\in M$. Finally, since $H$ is a finite Heyting algebra, it follows from Proposition \ref{P3.5} that $M^\bigstar$ becomes a Heyting module over $H$ and the canonical map
$M\longrightarrow M^{\bigstar\bigstar}$ becomes a morphism of Heyting modules. 

\end{proof}

\section{Coequalizers and Equalizers in $Heymod_H$}

Throughout this section, we let $H$ be a finite Heyting algebra. In other words (see Remark \ref{XR3.7}), this means that $H$
is a finite distributive lattice. We also recall from Proposition \ref{P3.5} that any distributive module over a finite Heyting algebra $H$
is canonically equipped with the structure of a Heyting module. In this section, we will study products, coproducts, coequalizers and equalizers in $Heymod_H$ in a manner analogous to \cite[$\S$ 3]{one}. In other words, we study elements of homological algebra of distributive modules over a finite distributive lattice, extending the results from \cite{one} for modules over the Boolean semifield $\mathbb B=\{0,1\}$.

\smallskip
It is clear that the object $0$ is both initial and final in $Heymod_H$. Our first aim is to show that $Heymod_H$ is a semiadditive
category, i.e., $Heymod_H$ has all finite biproducts (see \cite[VII.2]{ML}).

\begin{lem}\label{uL5.1}
(a) The category $Heymod_H$ contains all finite products.

\smallskip
(b) In the category $Heymod_H$, finite products are isomorphic to finite coproducts.
\end{lem}

\begin{proof}
(a) We consider Heyting modules $M$, $N$. Then, $M\times N=\{\mbox{$(m,n)$ $\vert$ $m\in M$, $n\in N$}\}$ becomes
a Heyting module with the operations
\begin{equation}
(m,n)\vee (m',n') =(m\vee m',n\vee n') \qquad c\wedge (m,n)=(c\wedge m,c\wedge n)
\end{equation} for $(m,n)$, $(m',n')\in M\times N$ and $c\in H$. We also notice that the canonical projections
$p_1:M\times N\longrightarrow M$, $p_2:M\times N\longrightarrow N$ as well as the inclusions $e_1:M\longrightarrow 
M\times N$, $e_1(m)=(m,0)$ and $e_2:N\longrightarrow M\times N$, $e_2(n)=(0,n)$ are morphisms of Heyting modules. 

\smallskip
Given morphisms $f:X\longrightarrow M$, $g:X\longrightarrow N$ in
$Heymod_H$, it is clear that $(f,g):X\longrightarrow M\times N$ given by $(f,g)(x)=(f(x),g(x))$ for each $x\in X$ 
is the unique morphism in $Heymod_H$ such that $p_1\circ (f,g)=f$ and $p_2\circ (f,g)=g$. Hence, $Heymod_H$
contains all finite products.

\smallskip
(b) We now suppose that we are given morphisms $f:M\longrightarrow Y$ and $g:N\longrightarrow Y$ in $Heymod_H$. Then, we define $h:M\times N\longrightarrow Y$ by setting $h(m,n)=f(m)\vee g(n)$. Since $(m,n)=(m,0)\vee (0,n)$, it is clear that $h$ is the unique morphism from
$M\times N$ to $Y$ such that $h\circ e_1=f$ and $h\circ e_2=g$. Hence, $M\times N$ is also the coproduct in
$Heymod_H$.
\end{proof}

\begin{thm}\label{uP5.2} Let $H$ be a finite Heyting algebra. Then, $Heymod_H$ is a semiadditive category. 
\end{thm}

\begin{proof}
This follows immediately from Lemma \ref{uL5.1} and the definition of a semi-additive category.
\end{proof}

\begin{rem}
\emph{If $\mathcal C$ is any category with zero objects and finite products, given objects $c_1$, $c_2\in \mathcal C$, there are canonical morphisms $(id,0):c_1\longrightarrow c_1\times c_2$ and 
$(0,id):c_2\longrightarrow c_1\times c_2$. If $\mathcal C$ also has finite coproducts, these morphisms together induce
a canonical morphism $\gamma_{12}:c_1\coprod c_2\longrightarrow c_1\times c_2$. Classically, a category is said
to be semiadditive if every such canonical morphism $\gamma_{12}:c_1\coprod c_2\longrightarrow c_1\times c_2$ is an isomorphism. However, a result of Lack \cite[Theorem 5]{Lack} shows that giving any family of  isomorphisms $c_1\coprod c_2\overset{\cong}{
\longrightarrow}c_1\times c_2$ that is natural in $c_1$, $c_2$ is enough to show that $\mathcal C$ is semiadditive. }
\end{rem} 

\smallskip
Our next aim is to construct the coequalizer and equalizer of morphisms in $Heymod_H$. By definition, a congruence relation  on some $M\in Heymod_H$
will be a Heyting submodule $R\subseteq M\times M$ which satisfies the following three conditions:

\smallskip
(1) For each $m\in M$, we have $(m,m)\in R$.

\smallskip
(2) For $m$, $m'\in M$ such that $(m,m')\in R$, we must have $(m',m)\in R$. 

\smallskip
(3) If $m$, $m'$ and $m''\in M$ are such that $(m,m')$, $(m',m'')\in M$, then $(m,m'')\in M$. 

\begin{lem}\label{vL5.4} If $R\subseteq M\times M$ is a congruence relation on $M$, then $M/R$ is a Heyting module.
\end{lem}

\begin{proof}
By definition, $M/R$ is the set of equivalence classes in $M$, where $m\sim m'$ if and only if $(m,m')\in R$. Given elements
$[m]$, $[n]\in M/R$ corresponding respectively to elements $m$, $n\in M$, we set
\begin{equation}\label{e5.2re}
[m]\vee [n]:=[m\vee n]\qquad c\wedge [m]:=[c\wedge m]
\end{equation} for each $c\in H$. Say $m\sim m'$ and $n\sim n'$ in $M$. Then, $(m,m')$, $(n,n')\in R$ and since $R$ is a submodule,
we must have $(m\vee n,m'\vee n')=(m,m')\vee (n,n')\in R$. It follows that $[m\vee n]=[m'\vee n']$. Since $R$ is a submodule, it also
follows that $(c\wedge m,c\wedge m')\in R$ for any $c\in H$ and hence $[c\wedge m]=[c\wedge m']$. Hence, the Heyting module structure on 
$M/R$ given in \eqref{e5.2re} is well-defined. 
\end{proof}

Given morphisms $f,g:L\longrightarrow M$ in $Heymod_H$, we now set $R_{(f,g)}$ to be the intersection of all congruence relations on $M$
containing the collection $\{(f(x),g(x))\}_{x\in L}$. 

\begin{thm}\label{vP5.5} (a) The canonical morphism $r:M\longrightarrow M/R_{(f,g)}$ is the coequalizer of the morphisms $f,g:L\longrightarrow M$
in $Heymod_H$.

\smallskip
(b) The coequalizer of $f,g:L\longrightarrow M$ is the quotient of $M$ over the congruence relation given by
\begin{equation}\label{5.3'}
R'_{(f,g)}=\{\mbox{$(m,n)\in M\times M$ $\vert$ $t(m)=t(n)$ for every $t:M\longrightarrow Q$ in $Heymod_H$ with $t\circ f=t\circ g$}\}
\end{equation}

\end{thm}

\begin{proof} (a) From the definition of $R_{(f,g)}$, it is clear that $r\circ f=r\circ g$. Further, if $s:M\longrightarrow P$ is a morphism such that 
$s\circ f=s\circ g$, we set 
\begin{equation}
R_s:=\{\mbox{$(m,n) \in  M\times M$ $\vert$ $s(m)=s(n)$}\}
\end{equation} It is clear that $R_s\subseteq M\times M$ gives a congruence relation on $M$ and that $(f(x),g(x))\in R_s$ for every $x\in L$. It follows that
$R_{(f,g)}\subseteq R_s$. Then, for any $m$, $n\in M$:
\begin{equation}\label{5.4eg}
r(m)=r(n)\quad\Rightarrow \quad (m,n)\in R_{(f,g)}\quad \Rightarrow \quad (m,n)\in R_s \quad \Rightarrow \quad s(m)=s(n)
\end{equation} As such, \eqref{5.4eg} shows that the morphism $s:M\longrightarrow P$ factors uniquely through $M/R_{(f,g)}$. 

\smallskip
(b) Since $R'_{(f,g)}$ is a congruence relation containing all the elements $\{(f(x),g(x))\}_{x\in L}$, it follows from the definition of
$R_{(f,g)}$ that $R_{(f,g)}\subseteq R'_{(f,g)}$. Conversely, consider some $(m,n)\in R'_{(f,g)}$. From part (a), we know that the canonical
morphism $r:M\longrightarrow M/R_{(f,g)}$ satisfies $r\circ f=r\circ g$. From \eqref{5.3'} it follows that $r(m)=r(n)$, i.e., $(m,n)\in R_{(f,g)}$. 
Hence, $R_{(f,g)}=R'_{(f,g)}$ and the result follows. 

\end{proof}

We also see that the equalizer $L'\hookrightarrow L$ of morphisms $f,g:L\longrightarrow M$ in $Heymod_H$ is given by
\begin{equation}\label{eqlz}
L':=\{\mbox{$x\in L$ $\vert$ $f(x)=g(x)$}\}
\end{equation} It is easily verified that $L'$ is a Heyting submodule of $L$. 

\smallskip
The next step is to consider coimages and images in the category $Heymod_H$. This will be done in a manner analogous to \cite[$\S$ 3]{one}, by considering 
``kernel pairs'' and ``cokernel pairs.'' 

\smallskip
We begin with a morphism $f:M\longrightarrow N$ in $Heymod_H$. Considering $M^2:=M\times M$ along with the canonical projections $p_1,p_2:M^2
\longrightarrow M$, we have morphisms $f\circ p_1,f\circ p_2:M^2\longrightarrow N$ in $Heymod_H$. The kernel pair $Ker_p(f)$ is now defined to be the equalizer
\begin{equation}\label{kerp}
Ker_p(f):=Eq\left(\xymatrix{
 M^2 \ar@<-.5ex>[r]_{f\circ p_2} \ar@<.5ex>[r]^{f\circ p_1} & N
}\right)
\end{equation} From the definition in \eqref{eqlz}, we know that $Ker_p(f)\subseteq M^2$. The coimage $Coim(f)$ is taken to be the coequalizer
\begin{equation}\label{coim}
Coim(f):=Coeq\left(Ker_p(f)\hookrightarrow \xymatrix{
 M^2 \ar@<-.5ex>[r]_{p_2} \ar@<.5ex>[r]^{p_1} & M
}\right)
\end{equation} The next result describes the coimage more explicitly.

\begin{thm}\label{usP5.6} (a) The submodule $Ker_p(f)\subseteq M$ defines a congruence relation on $M$.

\smallskip
(b) The coimage $Coim(f)$ is the quotient of $M$ over the equivalence relation $m\sim m'$ $\Leftrightarrow$ $f(m)=f(m')$. 

\smallskip
(c) The coimage $Coim(f)$ is isomorphic to the Heyting submodule $I_f:=\{\mbox{$f(m)$ $\vert$ $m\in M$}\}$ of $N$. 
\end{thm}

\begin{proof} Using \eqref{eqlz}, 
we see that an element $(m,m')\in M\times M$ lies in $Ker_p(f)$ if and only if $f(m)=f\circ p_1(m,m')=f\circ p_2(m,m')=f(m')$. It is clear
that this gives a congruence relation on $M$. This proves (a). By definition, the coequalizer $Coim(f)$ in \eqref{coim} is the quotient of $M$ over
the smallest congruence relation containing $Ker_p(f)$. Since $Ker_p(f)$ is already a congruence relation, it follows that $Coim(f)$ is 
the quotient of $M$ over the equivalence relation
\begin{equation}\label{5.9eqp}
m\sim m' \quad\Leftrightarrow\quad  (m,m')\in Ker_p(f)\quad\Leftrightarrow\quad f(m)=f(m')
\end{equation} This proves (b). The result of (c) is clear from \eqref{5.9eqp}. 

\end{proof} 

We now consider $N^2=N\times N$ along with the canonical inclusions $e_1,e_2:N\longrightarrow N^2$. Proceeding in a dual manner, the cokernel pair $Coker_p(f)$ is taken to be the coequalizer
\begin{equation}\label{cokerp}
Coker_p(f):=Coeq\left( \xymatrix{
 M\ar@<-.5ex>[r]_{e_2\circ f} \ar@<.5ex>[r]^{e_1\circ f} & N^2
}\right)
\end{equation} Further, the image $Im(f)$ is taken to be the equalizer
\begin{equation}\label{im}
Im(f):=Eq\left( \xymatrix{
 N\ar@<-.5ex>[r]_{e_2} \ar@<.5ex>[r]^{e_1} & N^2
}\longrightarrow Coker_p(f)\right)
\end{equation} The next result gives an explicit description of the image of a morphism in $Heymod_H$.

\begin{thm}\label{P5.7b}
Let $H$ be a finite Heyting algebra and $f:M\longrightarrow N$ be a morphism in $Heymod_H$. Then, the image of $f$ is given by
\begin{equation*}
Im(f)=\{\mbox{$n\in N$ $\vert$ $t_1(n)=t_2(n)$ $\forall$ $Q\in Heymod_H$, $\forall$ $t_1,t_2:N\longrightarrow Q$ such that
$t_1(f(m))=t_2(f(m))$ $\forall$ $m\in M$ }\}
\end{equation*}
In particular, if $i:K\hookrightarrow N$ is a monomorphism in $Heymod_H$, the image $\widetilde{K}=Im(i)$ is given by
the submodule
\begin{equation*}
\widetilde{K}=\{\mbox{$n\in N$ $\vert$ $t_1(n)=t_2(n)$ $\forall$ $Q\in Heymod_H$ $\forall$  $t_1,t_2:N\longrightarrow Q$ such that
$t_1(k)=t_2(k)$ $\forall$ $k\in K$ }\}
\end{equation*}
\end{thm}

\begin{proof}
From Proposition \ref{vP5.5} and the definition in \eqref{cokerp}, we see that $Coker_p(f)$ is the quotient of $N^2$ over the equivalence relation 
$(n_1,n_2)\sim (n_1',n_2')$ if $t(n_1,n_2)=t(n_1',n_2')$ for every $t:N^2\longrightarrow Q$ in $Heymod_H$ such that $t(f(m),0)=t(0,f(m))$ $\forall$ 
$m\in M$. Then, \eqref{im} shows that $Im(f)$ consists of all $n\in N$ such that $(n,0)\sim (0,n)\in N^2$. 

\smallskip
Unpacking this definition, we see that $Im(f)$ consists of all $n\in N$ such that $t(n,0)=t(0,n)$ for every $t:N^2\longrightarrow Q$ 
such that $t(f(m),0)=t(0,f(m))$ for every $m\in M$. We see that a  morphism $t:N^2\longrightarrow Q$ corresponds to two separate morphisms $t_1,t_2:N\longrightarrow
Q$ such that $t_1(n)=t(n,0)$ and $t_2(n)=t(0,n)$. The result is now clear. 
\end{proof}

For modules over a finite Heyting algebra, the following result replaces the usual (AB2) property (isomorphism of coimage and image) for abelian categories.

\begin{thm}\label{sP5.71}
Let $H$ be a finite Heyting algebra and let $f:M\longrightarrow N$ be a morphism in $Heymod_H$. Then,  we have $\widetilde{Coim(f)}=Im(f)$.
\end{thm}

\begin{proof} From Proposition \ref{P5.7b}, we see that $Im(f)=\widetilde{I_f}$, where $I_f=\{\mbox{$f(m)$ $\vert$ $m\in M$}\}$. From Proposition \ref{usP5.6}, we know that $Coim(f)=I_f$ and hence the result follows. 

\end{proof} 

In the case where $H$ is a finite Boolean algebra, we have an isomorphism between the coimage and the image. 

\begin{thm}\label{sP5.8}
Suppose that $H$ is a finite Boolean algebra and let $f:M\longrightarrow N$ be a morphism in $Heymod_H$. Then, we have ${Coim(f)}= Im(f)$.
\end{thm}

\begin{proof}
From Proposition \ref{P4.4}, it follows that $\widetilde{K}=K$ for any submodule $K\subseteq N$. Hence, $\widetilde{I_f}=I_f\subseteq N$
and we get $Im(f)=\widetilde{I_f}=I_f=Coim(f)$.
\end{proof} 

\section{Tensor products of Heyting modules and change of base}

We continue with $H$ being a finite Heyting algebra. In this section, we construct the tensor product $M\otimes_HN$ of Heyting modules
over $H$. For a study of tensor products of lattices and semilattices in the literature, we refer the reader
to \cite{AK}, \cite{Fraser}, \cite{Gra} and \cite{Shm}. 

\smallskip
For any set $S$, we define $Free_H(S)$ to be the collection of all functions $f:S\longrightarrow H$ of finite support, i.e., there are only
finitely many elements in $S$ such that $f(s)\ne 0$. It is easily seen that $Free_H(S)$ is a distributive module over $H$:
\begin{equation}\label{6xeq6.1}
(f\vee g)(s):=f(s)\vee g(s)\qquad (c\wedge f)(s):=c\wedge f(s)\qquad \forall\textrm{ }s\
\in S, \textrm{ }\forall\textrm{ }f,g\in Free_H(S)
\end{equation} Since $H$ is finite, this makes $Free_H(S)$ a Heyting module. For the sake of convenience, an element of
$Free_H(S)$ will be denoted by a formal sum $\sum c_i\wedge s_i$, where $s_i\in S$ and $c_i=f(s_i)\in H$.  

\smallskip
In particular, if $M$, $N\in Heymod_H$, then $Free_H(M\times N)$ consists of sums of the form 
$\sum c_i\wedge (m_i,n_i)$, where each $(m_i,n_i)\in M\times N$ and each $c_i\in H$.

\smallskip
As a set, we now define $M\otimes_HN$ to be the quotient of $Free_H(M\times N)$ over the equivalence relation generated by the following:
\begin{equation}\label{6xeq6.2}
\begin{array}{c}
\sum c_i\wedge (m_i,n_i) + c\wedge (0,n)=\sum c_i\wedge (m_i,n_i) = \sum c_i\wedge (m_i,n_i) + c\wedge (m,0)\\
\sum c_i\wedge (m_i,n_i) + c\wedge (m,n)+c\wedge (m',n)=\sum c_i\wedge (m_i,n_i) + c\wedge (m\vee m',n)\\ \sum c_i\wedge (m_i,n_i) +c\wedge (m,n)+c\wedge (m,n')=
\sum c_i\wedge (m_i,n_i) +c\wedge (m,n\vee n')\\
\sum c_i\wedge (m_i,n_i) +c\wedge (d\wedge m,n)=\sum c_i\wedge (m_i,n_i) +(c\wedge d)\wedge (m,n)=\sum c_i\wedge (m_i,n_i) +c\wedge (m,d\wedge n)\\
\end{array}
\end{equation} where $c$, $d$, $c_i\in H$, $m$, $m_i\in M$ and $n$, $n_i\in N$. From the relations in \eqref{6xeq6.2}, it is evident
that the $\vee$ operation on $Free_H(M\times N)$ as well as $c\wedge \_\_$ operation for each $c\in H$ descends
to $M\otimes_HN$, making it a distributive module over $H$ and hence a Heyting module. Given $m\in M$, $n\in N$, we will denote
by $m\otimes n$ the equivalence class of the element $(m,n)\in Free_H(M\times N)$ in $M\otimes_HN$.

\begin{defn}\label{6xD6.1}
Let $H$ be a finite Heyting algebra and let $M$, $N$, $P$ be Heyting modules. A bimorphism $f:M\times N\longrightarrow P$
is a map such that for fixed $m\in M$ and $n\in N$, the maps
\begin{equation*}
g_m:=f(m,\_\_):N\longrightarrow P\qquad h_n:=f(\_\_,n):M\longrightarrow P
\end{equation*} are morphisms of Heyting modules. 
\end{defn}

Analogous to  ordinary tensor products of modules, we will now see that $M\otimes_HN$ represents bimorphisms from $M\times N$
to $P$.  For the similar notion of bimorphisms of join semilattices, see \cite{Gra}. 

\begin{thm}\label{6xP6.2}
Let $M$, $N$ and $P$ be Heyting modules. Then,  for each bimorphism
$f:M\times N\longrightarrow P$, there is a unique morphism $f':M\otimes_HN\longrightarrow P$ in $Heymod_H$ such that
$f'(m\otimes n)=f(m,n)$. 
\end{thm}

\begin{proof}
We consider a bimorphism $f:M\times N\longrightarrow P$. The morphism $f':M\otimes_HN\longrightarrow P$ is defined by
taking the equivalence class of the element $\sum c_i\wedge (m_i,n_i)$ to $\sum c_i\wedge f(m_i,n_i)\in P$. From the relations 
in \eqref{6xeq6.2} and the fact that $f$ is a bimorphism, we see that $f'$ is well-defined. It is also clear that $f'$ is a morphism of
Heyting modules. 

\smallskip
Since $m\otimes n$ is the equivalence class of the element $(m,n)\in Free_H(M\times N)$ in $M\otimes_HN$,  the definition
gives $f'(m\otimes n)=f(m,n)$. In general, if $g:M\otimes_HN\longrightarrow P$ is a morphism of Heyting modules such that
$g(m\otimes n)=f(m,n)$, then $g$ must take the equivalence
class of  $\sum c_i\wedge (m_i,n_i)$ to $\sum c_i\wedge f(m_i,n_i)$. Hence, $f'$ must be unique. 
\end{proof}

\begin{cor}\label{6xC6.3}
Given Heyting modules $M$, $N$ and $P$, we have isomorphisms:
\begin{equation}
\begin{array}{c}
(M\otimes_HN)\cong (N\otimes_HM)\\
(M\otimes_HN)\otimes_HP\cong M\otimes_H(N\otimes_HP)\\
\end{array}
\end{equation}
\end{cor}

\begin{proof}
Using  the description of morphisms from the tensor product in Proposition \ref{6xP6.2} and applying 
Yoneda lemma to the category $Heymod_H$, the result follows.  
\end{proof}

\smallskip
 Given morphisms $f$, $g:M\longrightarrow N$
in $Heymod_H$ and any $c\in H$, we define
\begin{equation}\label{6ins}
(f\vee g)(m):=f(m)\vee g(m)\qquad (c\wedge f)(m):=c\wedge f(m)\qquad \forall\textrm{ }m\
\in M
\end{equation} It is easily verified that this makes $Heymod_H(M,N)$ into a distributive module over $H$. Since $H$ is finite,
we get $Heymod_H(M,N)\in Heymod_H$. We will often write $f\vee g$ as $f+g$. 

\begin{thm}\label{6xP6.4}
For any $N\in Heymod_H$, the functor $\_\_\otimes_HN:Heymod_H\longrightarrow Heymod_H$ is left adjoint to the functor $
Heymod_H(N,\_\_):Heymod_H\longrightarrow Heymod_H$. 
\end{thm}

\begin{proof}
We consider $M$, $N$, $P\in Heymod_H$ and a morphism $f:M\longrightarrow Heymod_H(N,P)$ in $Heymod_H$. We define $g:M\times N\longrightarrow
P$ by setting $g(m,n):=f(m)(n)$. Then, for each fixed $m\in M$, $g(m,\_\_)=f(m)$ is a morphism of Heyting modules from $N$ to $P$. If we fix $n\in N$, it follows from the definitions in \eqref{6ins} and the fact that $f$ is a morphism of Heyting modules that
\begin{equation*}
\begin{array}{c}
g(m',n)\vee g(m'',n)=f(m')(n)\vee f(m'')(n)=(f(m')\vee f(m''))(n)=f(m'\vee m'')(n)=g(m'\vee m'',n)\\ g(c\wedge m',n)=f(c\wedge m')(n)=(c\wedge f(m'))(n)=c\wedge f(m')(n)=c\wedge g(m',n)
\end{array}
\end{equation*} for any $m'$, $m''\in M$ and $c\in H$. It follows that $g:M\times N\longrightarrow P$ is a bimorphism. Since $f$ and $g$ completely determine each other, it now follows from Proposition \ref{6xP6.2} that we have an isomorphism
$Heymod_H(M\otimes_HN,P)\cong Heymod_H(M,Heymod_H(N,P))$.
\end{proof}

We are now ready to consider base extensions of Heyting modules.

\begin{thm}\label{6xP6.5}
Let $f:H\longrightarrow H'$ be a morphism between finite Heyting algebras, i.e., $f$ is a morphism of the underlying distributive lattices. 
Then, there is an `extension of scalars' along $f$, i.e., $f$ induces a functor $\_\_\otimes_HH':Heymod_H\longrightarrow Heymod_{H'}$.
\end{thm}

\begin{proof}
If $f:H\longrightarrow H'$ is a morphism of Heyting algebras, then $H'\in Heymod_H$. Accordingly, for any $M\in Heymod_H$, we can form
the tensor product $M\otimes_HH'$. For any element in $M\otimes_HH'$ represented by $\sum c_i\wedge (m_i,h'_i)$ and any $h'\in H'$, we set
$(\sum c_i\wedge (m_i,h'_i))\wedge h':=\sum c_i\wedge (m_i,h'_i\wedge h')$. It may be verified easily that this operation gives $M\otimes_HH'\in Heymod_{H'}$. 
\end{proof}

On the other hand, given a morphism $f:H\longrightarrow H'$ between finite Heyting algebras, there is an obvious
restriction functor $Res^{H'}_H:Heymod_{H'}\longrightarrow Heymod_H$. We record the following observation.

\begin{thm}\label{6xP6.6}
Let $f:H\longrightarrow H'$ be a morphism between finite Heyting algebras. Let $N'\in Heymod_{H'}$ and set $N:=Res^{H'}_H(N')$. Then, we have
$
f((n_1\rightarrow n_2)_{N})\leq (n_1\rightarrow n_2)_{N'}$ for all $n_1,n_2\in N'
$.
\end{thm}

\begin{proof} The $H$-module structure on $N=Res^{H'}_H(N')$ is given by $(c,n)\mapsto f(c)\wedge n$ for every $c\in H$ and $n\in N$. From the construction in Proposition \ref{P3.5}, we now see  that $(n_1\rightarrow n_2)_N\in H$ is the supremum of all elements $c\in H$ such that
$f(c)\wedge n_1\leq n_2$. Also, $(n_1\rightarrow n_2)_{N'}\in H'$ is the supremum of all elements $c'\in H'$ such that $c'\wedge n_1\leq n_2$. In particular, this means that if $c\in H$ is such that $f(c)\wedge n_1\leq n_2$, we will have $f(c)\leq (n_1\rightarrow n_2)_{N'}$. Since $H$ is finite and $f$ preserves finite joins,
it is now clear that $
f((n_1\rightarrow n_2)_{N})\leq (n_1\rightarrow n_2)_{N'}$ for all $n_1,n_2\in N'
$.

\end{proof}

\begin{thm}\label{6xP6.7}
Let $f:H\longrightarrow H'$ be a morphism between finite Heyting algebras. Then, the functor $\_\_\otimes_HH':Heymod_H
\longrightarrow Heymod_{H'}$ is left adjoint to the restriction $Res^{H'}_H:Heymod_{H'}\longrightarrow Heymod_H$. 
\end{thm}

\begin{proof}
We consider $M\in Heymod_H$, $N\in Heymod_{H'}$ and a morphism $g_1:M\otimes_HH'\longrightarrow N$ in $Heymod_{H'}$. Then, $Res^{H'}_H(g_1)$
is  a morphism in $Heymod_H$ and we compose it with  $1_M\otimes_Hf:M\otimes_HH\longrightarrow M\otimes_HH'$ to obtain a morphism
$M=M\otimes_HH\longrightarrow Res_H^{H'}(N)$ in $Heymod_H$. 

\smallskip
Conversely, suppose that we are given a morphism $g_2:M\longrightarrow Res^{H'}_H(N)$ in $Heymod_H$. Then, $g_2$ induces a morphism
$g_2\otimes_HH': M\otimes_HH'\longrightarrow Res^{H'}_H(N)\otimes_HH'$ in $Heymod_{H'}$. Since $N\in Heymod_{H'}$, we have a canonical morphism $Res^{H'}_H(N)\otimes_HH'\longrightarrow N$ in $Heymod_{H'}$. Composing this with $g_2\otimes_HH'$, we obtain a morphism
$M\otimes_HH'\longrightarrow N$ in $Heymod_{H'}$. It is clear that these two associations are inverse to each other and this proves the result. 
\end{proof}

\section{The Kleisli category $\mathbf{Heymod_H^2}$}

We continue with $H$ being a finite Heyting algebra. From Section 5,  we see that  the sequence 
\begin{equation}\label{hexact}
\xymatrix{
 Ker_p(f)\ar@<-.5ex>[r] \ar@<.5ex>[r]  & M
}\overset{f}{\longrightarrow} \xymatrix{
 N\ar@<-.5ex>[r] \ar@<.5ex>[r] & Cok_p(f)
}
\end{equation}
replaces the usual exact sequence involving the kernel and the cokernel in an abelian category. In a manner similar to \cite[$\S$ 3,4]{one}, the sequence in 
\eqref{hexact} suggests that we study the category $Heymod_H^2$ which has the same objects as $Heymod_H$, but  each morphism $M\longrightarrow N$ in $Heymod_H^2$ is a pair $(f,g):M\longrightarrow N$ of morphisms in $Heymod_H$. The composition of morphisms in $Heymod_H^2$ follows from the intuition that the pair $(f,g)$ should play the role of ``$f-g$.''  As such, the composition law for morphisms in $Heymod_H^2$ is given by:
\begin{equation}\label{composition}
(f',g')\circ (f,g)=(f'\circ f+g'\circ g,f'\circ g+g'\circ f)
\end{equation}  The category $Heymod_H$ is canonically embedded in $Heymod_H^2$ by taking any morphism $f$ to the pair $(f,0)$, giving 
a functor $\kappa_H:Heymod_H\hookrightarrow Heymod_H^2$.

\smallskip

The construction of $Heymod_H^2$ may be understood more categorically as follows: we recall the notion of a comonad on a category $\mathcal C$. 

\begin{defn}\label{comonad} (see, for instance, \cite[p 219]{Borceux} Given a category $\mathcal C$, the composition of functors determines a monoidal (but not necessarily symmetric monoidal) structure on the category $Fun(\mathcal C,\mathcal C)$ of endofunctors $\mathcal C\longrightarrow\mathcal C$. A comonad  on $\mathcal C$ is a comonoid in the monoidal category $Fun(\mathcal C,\mathcal C)$.

\smallskip
More explicitly, a comonad on $\mathcal C$ is a triple $(\bigperp,\delta,\epsilon)$ consisting of a functor $\bigperp:\mathcal C
\longrightarrow\mathcal C$ and natural transformations
\begin{equation}\label{comonad2}
\delta: \bigperp\longrightarrow \bigperp\circ\bigperp \qquad \epsilon : \bigperp\longrightarrow id
\end{equation} satisfying the conditions for coassociativity and counity respectively. 

\end{defn}

\begin{thm}\label{P6.2} Let $H$ be a finite Heyting algebra. Then, the endofunctor $\bigperp :Heymod_H\longrightarrow Heymod_H$ defined by taking any object $M\in Heymod_H$ to $M^2$ and any morphism $f:M\longrightarrow N$ to $(f,f):M^2
\longrightarrow N^2$ determines a comonad on $Heymod_H$.
\end{thm}

\begin{proof}
The natural transformations $\epsilon : \bigperp\longrightarrow id$ and $\delta: \bigperp\longrightarrow \bigperp\circ\bigperp$
are defined by setting for each $M\in Heymod_H$:
\begin{equation}\label{eq6.4}
\begin{array}{c}
\epsilon(M):\bigperp M=M^2\longrightarrow M \qquad (m,n)\mapsto m\\
\delta(M): \bigperp M=M^2 \longrightarrow \bigperp\circ\bigperp M=(M^2)^2 \qquad (m,n)\mapsto (m,n,n,m)\\
\end{array}
\end{equation} It is clear that the morphisms in \eqref{eq6.4} lie in $Heymod_H$. The counit property of $\epsilon$ follows from the commutativity of the following diagram
\begin{equation*}
\begin{CD}
(m,n)\in  \bigperp M @>\delta(M)>> \bigperp (\bigperp M)\ni (m,n,n,m) \\
@V\delta(M)VV @VV\epsilon(\bigperp M)V\\
(m,n,n,m)\in \bigperp (\bigperp M) @>\bigperp(\epsilon(M))>> \bigperp M \ni (m,n)\\
\end{CD}
\end{equation*} The coassociativity property of $\delta$ follows from the commutativity of the following diagram
\begin{equation*}
\begin{CD}
(m,n)\in \bigperp M @>\delta(M)>> \bigperp (\bigperp M)\ni ((m,n),(n,m)) \\
@V\delta(M)VV @VV\delta(\bigperp M)V\\
((m,n),(n,m))\in \bigperp (\bigperp M)@>\bigperp(\delta(M))>> \bigperp\bigperp\bigperp M\ni ((m,n),(n,m),(n,m),(m,n))=((m,n,n,m),(n,m,m,n)))\\
\end{CD}
\end{equation*}
\end{proof}

By definition, the Kleisli category $Kl_{\bigperp}(\mathcal C)$ of a comonad $\bigperp$  (see \cite[p 192]{Borceux})  on a category $\mathcal C$ is constructed
as follows: the objects of $Kl_{\bigperp}(\mathcal C)$ are the same as those of $\mathcal C$ and the morphism sets are defined by setting:
\begin{equation}
Kl_{\bigperp}(C_1,C_2):=\mathcal C(\bigperp C_1,C_2)\qquad\forall\textrm{ }C_1,C_2\in Ob(\mathcal C)
\end{equation}  Given morphisms $f\in Kl_{\bigperp}(C_1,C_2)=\mathcal C(\bigperp C_1,C_2)$ and $g\in Kl_{\bigperp}(C_2,C_3)
=\mathcal C(\bigperp C_2,C_3)$, the composition is given by
\begin{equation}\label{eq6.6}
\left(
\begin{CD}\bigperp C_1 @>\delta(C_1)>>\bigperp(\bigperp C_1)
@>{\bigperp f}>> \bigperp C_2 @>g>> C_3\end{CD}\right)\in \mathcal C(\bigperp C_1,C_3)=Kl_{\bigperp}(C_1,C_3)
\end{equation}

\begin{thm}\label{P6.3}
For a finite Heyting algebra $H$, the  Kleisli category of the comonad $\bigperp$ is given by $Heymod_H^2$. 
\end{thm}

\begin{proof}
By definition, a morphism from $M$ to $N$ in $Kl_{\bigperp}(Heymod_H)$ consists of a morphism $M^2\longrightarrow N$
in $Heymod_H$, i.e., a pair $(f,g)$ of morphisms from $M$ to $N$ in $Heymod_H$. We consider a pair $(f',g')$ of morphisms from
$N$ to $P$ in $Heymod_H$. We calculate  $(f',g')\circ (f,g)$ as per the composition law for the Kleisli category in \eqref{eq6.6}. 

\smallskip
For this, we choose $(m,n)\in M^2=\bigperp M$. Then, we know that $\delta(M)(m,n)=(m,n,n,m)\in (M^2)^2$. From the morphism 
$(f,g):M^2\longrightarrow N$ which takes $(m,n)\mapsto f(m)\vee g(n)$, we obtain 
\begin{equation}
\bigperp (f,g)((m,n),(n,m))=((f,g),(f,g))((m,n),(n,m))=(f(m)\vee g(n),f(n)\vee g(m))\in N^2=\bigperp N
\end{equation} Finally, the pair $(f',g')$ takes $(f(m)\vee g(n),f(n)\vee g(m))\in N^2=\bigperp N$ to 
$(f'\circ f)(m)\vee (f'\circ g)(n)\vee (g'\circ f)(n)\vee (g'\circ g)(m)\in P$.  It is now clear that the composition in the Kleisli category
$Kl_{\bigperp}(Heymod_H)$ is identical to the composition in the category $Heymod_H^2$ described in \eqref{composition}. 
\end{proof}

The kernel pair of a morphism$  \xymatrix{
 M \ar@<-.5ex>[r]_{g} \ar@<.5ex>[r]^{f} & N
}$ in $Heymod_H^2$ is now defined by setting
\begin{equation}\label{kerp1}
Ker_p(f,g):=\{\mbox{$(m_1,m_2)\in M\times M$ $\vert$ $f(m_1)\vee g(m_2)=f(m_2)\vee g(m_1)$ }\}
\end{equation} If $g=0$, it is clear that \eqref{kerp1} recovers the notion of the kernel pair in \eqref{kerp}. 

\begin{lem}\label{L6.4}
Given a morphism $\xymatrix{
 M \ar@<-.5ex>[r]_{g} \ar@<.5ex>[r]^{f} & N
}$ in $Heymod_H^2$, $Ker_p(f,g)\subseteq M\times M$ is a Heyting submodule. 
\end{lem}

\begin{proof} Given $(m_1,m_2)$, $(m_1',m_2')\in Ker_p(f,g)$ we see that
\begin{equation}
\begin{array}{ll}
f(m_1\vee m_1')\vee g(m_2\vee m_2')=&(f(m_1)\vee g(m_2))\vee (f(m_1')\vee g(m_2'))\\
&=(f(m_2)\vee g(m_1))\vee (f(m_2')\vee g(m_1'))\\
&=f(m_2\vee m_2')
\vee g(m_1\vee m_1')\\
\end{array}
\end{equation} It follows that $(m_1,m_2)\vee (m_1',m_2')\in Ker_p(f,g)$. Also, for any $c\in  H$, we see that
\begin{equation}
f(c\wedge m_1)\vee g(c\wedge m_2)=c\wedge (f(m_1)\vee g(m_2))=c\wedge (f(m_2)\vee g(m_1))=f(c\wedge m_2)\vee g(c\wedge m_1)
\end{equation} and hence $(c\wedge m_1,c\wedge m_2)\in Ker_p(f,g)$. 

\end{proof}

While Lemma \ref{L6.4} shows that $Ker_p(f,g)$ is a Heyting submodule of $M\times M$, it should be pointed out that unlike the case of $Ker_p(f)$ in Proposition \ref{usP5.6}, $Ker_p(f,g)
\subseteq M\times M$ does not define a congruence relation on $M$. Also, $Ker_p(f,g)$ defined in \eqref{kerp1} is not an equalizer unlike 
\eqref{kerp}.  

\smallskip Being a submodule of $M\times M$, $Ker_p(f,g)$ is equipped with two canonical morphisms to $M$, which determine
a morphism $\xymatrix{
 Ker_p(f,g) \ar@<-.5ex>[r]_(0.6){k_2} \ar@<.5ex>[r]^(0.6){k_1} & M
}$ in $Heymod_H^2$. 
The next result  gives us something resembling a universal property for $Ker_p(f,g)$. For this, we note that the intuition of a morphism 
$(f,g)$
in $Heymod_H^2$ corresponding to ``$f-g$'' suggests that a composition $(f',g')\circ (f,g)$ in $Heymod_H^2$  ``corresponds to zero'' if and only
if $f'\circ f+g'\circ g=f'\circ g+g'\circ f$.

\begin{thm}\label{P6.5}
Let $\xymatrix{
 M \ar@<-.5ex>[r]_{g} \ar@<.5ex>[r]^{f} & N
}$ be a morphism in $Heymod_H^2$. Then the morphisms 
\begin{equation}
\xymatrix{
Ker_p(f,g) \ar@<-.5ex>[r]_(0.6){k_2} \ar@<.5ex>[r]^(0.6){k_1} & M \ar@<-.5ex>[r]_{g} \ar@<.5ex>[r]^{f} & N
}
\end{equation} 
in $Heymod_H^2$ satisfy $f\circ k_1+g\circ k_2=g\circ k_1+f\circ k_2$.  Further, if there exists a morphism $\xymatrix{L\ar@<-.5ex>[r]_{h_2} \ar@<.5ex>[r]^{h_1}&M}$ in $Heymod_H^2$ satisfying $f\circ h_1+g\circ h_2=g\circ h_1+f\circ h_2$, then $(h_1,h_2)$ factors through $(k_1,k_2)$.
\end{thm}

\begin{proof}
If $(m,m')\in Ker_p(f,g)$, then $k_1(m,m')=m$ and $k_2(m,m')=m'$. It follows from the definition in \eqref{kerp1} that
$f(m)\vee g(m')=f(m')\vee g(m)$ and hence $f\circ k_1+g\circ k_2=g\circ k_1+f\circ k_2$. 

\smallskip
From the definition in \eqref{kerp1}, it is also clear that for any $l\in L$, the element $(h_1(l),h_2(l))\in M\times M$ actually lies
in $Ker_p(f,g)$. This gives us a morphism $L\longrightarrow Ker_p(f,g)$ in $Heymod_H$ and hence a morphism
$\xymatrixcolsep{3pc}\xymatrix{L\ar@<-.5ex>[r]_(0.3){0} \ar@<.5ex>[r]^(0.3){(h_1,h_2)}& Ker_p(f,g)}$ in $Heymod_H^2$. We now have
the composition
\begin{equation}
(k_1,k_2)\circ ((h_1,h_2),0)=(k_1\circ (h_1,h_2),k_2\circ (h_1,h_2))=(h_1,h_2)
\end{equation} in $Heymod_H^2$, which proves the result. 
\end{proof}

Given a morphism  $\xymatrix{L\ar@<-.5ex>[r]_{f_2} \ar@<.5ex>[r]^{f_1}&M}$ in $Heymod_H^2$, we now set
\begin{equation}\label{im2}
I(f_1,f_2):=\{\mbox{$(f_1(x)\vee f_2(y),f_1(y)\vee f_2(x))$ $\vert$ $x$, $y\in L$}\}\subseteq M\times M
\end{equation} It is evident that $I(f_1,f_2)$ is a Heyting submodule of $M\times M$.  

\begin{lem}\label{L6.6}
Consider a composition of morphisms in $Heymod_H^2$ as follows:
\begin{equation}\label{eq6.14}
\xymatrix{L\ar@<-.5ex>[r]_{f_2} \ar@<.5ex>[r]^{f_1}&M\ar@<-.5ex>[r]_{g_2} \ar@<.5ex>[r]^{g_1}&N}
\end{equation} Then, we have:
\begin{equation}
I(f_1,f_2)\subseteq Ker_p(g_1,g_2)\quad\Leftrightarrow \quad g_1\circ f_1+g_2\circ f_2=g_1\circ f_2+g_2\circ f_1
\end{equation}
\end{lem}

\begin{proof}
We see that 
\begin{equation}\label{eq6.16}
\begin{array}{l}
I(f_1,f_2)\subseteq Ker_p(g_1,g_2)\\
\Leftrightarrow (f_1(x)\vee f_2(y),f_1(y)\vee f_2(x)) \in Ker_p(g_1,g_2)\\
\Leftrightarrow g_1(f_1(x)\vee f_2(y))\vee g_2(f_1(y)\vee f_2(x))=g_1(f_1(y)\vee f_2(x))\vee g_2(f_1(x)\vee f_2(y))\\
\Leftrightarrow  (g_1\circ f_1+g_2\circ f_2)(x)\vee (g_1\circ f_2+g_2\circ f_1)(y)=(g_1\circ f_2+g_2\circ f_1)(x)\vee (g_1\circ f_1+g_2\circ f_2)(y)\\
\end{array}
\end{equation} for all $x$, $y\in L$. Then, $I(f_1,f_2)\subseteq Ker_p(g_1,g_2)\Rightarrow g_1\circ f_1+g_2\circ f_2=g_1\circ f_2+g_2\circ f_1$ follows by setting
$y=0$ in \eqref{eq6.16}. The other implication is also clear from \eqref{eq6.16}. 
\end{proof}

We notice here that with a composition as in \eqref{eq6.14}
both $Ker_p(g_1,g_2)\subseteq M\times M$ and  $I(f_1,f_2)\subseteq M\times M$ are symmetric submodules of $M\times M$, i.e., they contain
an ordered pair $(m,m')$ if and only if they also contain $(m',m)$. We are now ready to define strict exactness in $Heymod_H^2$ in a manner parallel to \cite[Definition 4.4]{one}. 

\begin{defn}\label{D6.7}
A sequence of morphisms $\xymatrix{L\ar@<-.5ex>[r]_{f_2} \ar@<.5ex>[r]^{f_1}&M\ar@<-.5ex>[r]_{g_2} \ar@<.5ex>[r]^{g_1}&N}$ in $Heymod_H^2$ is strict exact at $M$ if $I(f_1,f_2)+\Delta_M=Ker_p(g_1,g_2)$. Here, $\Delta_M=\{\mbox{$(m,m)$ $\vert$ $m\in M$}\}\subseteq M\times M$ is the diagonal submodule of $M\times M$.
\end{defn}

\begin{thm}\label{P6.8} (a) A morphism $\xymatrix{L\ar@<-.5ex>[r]_{f_2} \ar@<.5ex>[r]^{f_1}&M}$ in $Heymod_H^2$ is a monomorphism
if and only if the induced map $L^2\longrightarrow M^2$ given by $(x,y)\mapsto (f_1(x)\vee f_2(y),f_1(y)\vee f_2(x))$ is injective. 

\smallskip
(b) A monomorphism $\xymatrix{L\ar@<-.5ex>[r]_{f_2} \ar@<.5ex>[r]^{f_1}&M}$ in $Heymod_H^2$ induces a strict exact sequence
$\xymatrix{0\ar@<-.5ex>[r]_{0} \ar@<.5ex>[r]^{0}&L\ar@<-.5ex>[r]_{f_2} \ar@<.5ex>[r]^{f_1}&M}$ in $Heymod_H^2$. 

\smallskip
(c) A sequence $\xymatrix{0\ar@<-.5ex>[r]_{0} \ar@<.5ex>[r]^{0}&L\ar@<-.5ex>[r]_{f_2} \ar@<.5ex>[r]^{f_1}&M}$ in $Heymod_H^2$ is strict
exact at $L$ if and only if 
\begin{equation}
f_1(x)\vee f_2(y)=f_1(y)\vee f_2(x)\qquad\Leftrightarrow \qquad x=y
\end{equation}

\end{thm}

\begin{proof}
If $L$ is any Heyting module, each element $l\in L$ determines a morphism $H\longrightarrow L$, $x\mapsto x\wedge l$  in $Heymod_H$. The rest of the proof is analogous to that of \cite[Proposition 4.6]{one} and \cite[Proposition 4.10]{one}. 
\end{proof}

We now come to the epimorphisms in $Heymod_H^2$ and the corresponding strict exact sequences. 

\begin{thm}\label{P7.9}
Let $\begin{CD} M@>\phi=(f,g)>> N@>>> 0\end{CD}$ be a sequence of morphisms in $Heymod_H^2$.  

\smallskip
(a) The following are equivalent:

\smallskip
$\hspace{0.2in}$(1) The sequence $\begin{CD} M@>\phi=(f,g)>> N@>>> 0\end{CD}$  is strictly exact at $N$.

\smallskip
$\hspace{0.2in}$(2) $\{\mbox{$f(x)\vee g(y)$ $\vert$ $f(y)\vee g(x)=0$, $x$, $y\in M$}\}=N$.

\smallskip
$\hspace{0.2in}$(3) $I(f,g)=N\times N$. 

\smallskip
(b) If  the sequence $\begin{CD} M@>\phi=(f,g)>> N@>>> 0\end{CD}$  is strictly exact at $N$, then the morphism 
$\phi=(f,g)$ is an epimorphism in $Heymod_H^2$.

\end{thm}

\begin{proof}
The proof of (a) is analogous to that of \cite[Proposition 4.5]{one}.  For (b), we proceed as follows: if  the sequence $\begin{CD} M@>\phi=(f,g)>> N@>>> 0\end{CD}$  is strictly exact at $N$, we have $I(f,g)=N\times N$. Explicitly speaking, this means that
\begin{equation}\label{eq7.17}
I(f,g):=\{\mbox{$(f(x)\vee g(y),f(y)\vee g(x))$ $\vert$ $x$, $y\in M$}\}= N\times N
\end{equation} Let $\psi=(\psi_1,\psi_2):N\longrightarrow P$ and 
$\psi'=(\psi'_1,\psi'_2):N\longrightarrow P$ be morphisms in $Heymod_H^2$ such that $\psi\circ \phi=\psi'\circ \phi$. Writing this out
explicitly, we get
\begin{equation}\label{eq7.18}
\begin{array}{c}
\psi_1\circ f+\psi_2\circ g=\psi'_1\circ f+\psi'_2\circ g \qquad 
\psi_1\circ g+\psi_2\circ f=\psi'_1\circ g+\psi'_2\circ f\\
\end{array}
\end{equation} The morphisms
$\psi$ and $\psi'$ induce morphisms $\tilde\psi=(\psi_1,\psi_2):N^2\longrightarrow P$ and $\tilde\psi'=(\psi'_1,\psi'_2):N^2\longrightarrow P$ in $Heymod_H$ given by
\begin{equation}\label{eq7.19}
\tilde\psi(z,w)=\psi_1(z)\vee \psi_2(w)\qquad \tilde\psi'(z,w)=\psi'_1(z)\vee \psi'_2(w)\qquad\forall\textrm{ }z,w\in N
\end{equation} For elements $x$, $y\in M$, \eqref{eq7.19} now gives
\begin{equation}\label{eq7.20}
\begin{array}{c}
\tilde\psi(f(x)\vee g(y),f(y)\vee g(x))=\psi_1(f(x)\vee g(y))\vee \psi_2(f(y)\vee g(x)) \\
\tilde\psi'(f(x)\vee g(y),f(y)\vee g(x))=\psi'_1(f(x)\vee g(y))\vee \psi'_2(f(y)\vee g(x)) \\
\end{array}
\end{equation} Applying \eqref{eq7.18}, we obtain $\tilde\psi|I(f,g)=\tilde\psi'|I(f,g)$. Since $I(f,g)=N\times N$, this gives
$\tilde\psi=\tilde\psi':N\times N\longrightarrow P$. In particular, $\psi_1(z)=\tilde\psi(z,0)=\tilde\psi'(z,0)=\psi'_1(z)$ and
$\psi_2(w)=\tilde\psi(0,w)=\tilde\psi'(0,w)=\psi'_2(w)$ for $z$, $w\in N$, i.e., $\psi=\psi'$. Hence, $\phi$ is an epimorphism
in $Heymod_H^2$.  
\end{proof}

\begin{thm}\label{P7.10} Let $H$ be a finite Boolean algebra. Then, the sequence $\begin{CD} M@>\phi=(f,g)>> N@>>> 0\end{CD}$ in
$Heymod_H^2$ is strictly exact at $N$ if and only if the morphism 
$\phi=(f,g)$ is an epimorphism in $Heymod_H^2$. 
\end{thm}

\begin{proof}
The ``only if part'' of this result already follows from Proposition \ref{P7.9}. For the ``if part,'' we maintain the notation from the proof 
of Proposition \ref{P7.9}. Let  $\tilde\psi=(\psi_1,\psi_2):N^2\longrightarrow P$ and $\tilde\psi'=(\psi'_1,\psi'_2):N^2\longrightarrow P$
be morphisms in $Heymod_H$  satisfying $\tilde\psi|I(f,g)=\tilde\psi'|I(f,g)$. Putting $y=0$ in \eqref{eq7.20}, we get 
\begin{equation}\label{eq7.21}
(\psi_1\circ f+\psi_2\circ g)(x)=\tilde\psi(f(x),g(x))=\tilde\psi'(f(x),g(x))=(\psi'_1\circ f+\psi'_2\circ g)(x)
\end{equation} Similarly, putting $x=0$ in \eqref{eq7.20} gives $\psi_1\circ g+\psi_2\circ f=\psi'_1\circ g+\psi'_2\circ f$. This means that
$\psi\circ \phi=\psi'\circ \phi$ in $Heymod_H^2$, where $\psi$ is given by the pair $(\psi_1,\psi_2)$ and $\psi'$ by the pair
$(\psi'_1,\psi'_2)$. Since $\phi$ is an epimorphism in $Heymod_H^2$, we obtain $\psi=\psi'$. Then, $\psi_1=\psi'_1$ and 
$\psi_2=\psi'_2$ and hence $\tilde\psi=\tilde\psi'$. Since $H$ is a finite Boolean algebra, we may now apply Proposition \ref{P4.4}
to prove that $I(f,g)=N\times N$.
\end{proof}

After monomorphisms and epimorphisms, we have to treat the isomorphisms in $Heymod_H^2$. 

\begin{thm}\label{P7.11}
A sequence $\xymatrix{0\ar@<-.5ex>[r]_{0} \ar@<.5ex>[r]^{0}& M\ar@<-.5ex>[r]_{g} \ar@<.5ex>[r]^{f}&N
\ar@<-.5ex>[r]_{0} \ar@<.5ex>[r]^{0}&0}$ is strict exact in $Heymod_H^2$ if and only if $\xymatrix{M\ar@<-.5ex>[r]_{g} \ar@<.5ex>[r]^{f}&N}$ is an isomorphism
in $Heymod_H^2$. Further, such a strict exact sequence corresponds to an isomorphism $h:M\overset{\cong}{\longrightarrow}N$
in $Heymod_H$ and a unique decomposition $N=N_1\times N_2$ such that $f$ and $g$ are induced respectively by the canonical projections 
$N_1\times N_2\longrightarrow N_1$ and $N_1\times N_2\longrightarrow N_1$ as follows
\begin{equation}
f:M\overset{h}{\longrightarrow}N=N_1\times N_2\longrightarrow N_1 \hookrightarrow N \qquad g:M\overset{h}{\longrightarrow}N=N_1\times N_2\longrightarrow N_2 \hookrightarrow N 
\end{equation}
\end{thm}

\begin{proof}
This may be proved in a manner similar to \cite[Proposition 4.11 \& Proposition 4.12]{one}.
\end{proof}

From the definition in \eqref{6ins}, we know that for any $M$, $M'\in Heymod_H$, the morphisms in $Heymod_H$ from $M$
to $M'$ form a Heyting module. Fix $M\in Heymod_H$. Considering products in $Heymod_H$, it  follows that the functor given by the association
\begin{equation}\label{eq7.24}
 N\mapsto Heymod_H^2(M,N)=Heymod_H(M,N)\times Heymod_H(M,N)
\end{equation} for each $N\in Heymod_H^2$ determines a covariant functor $Heymod_H^2(M,\_\_):Heymod_H^2\longrightarrow
Heymod_H$.

\smallskip
Given $N\subseteq M$ in $Heymod_H$, we now define
\begin{equation}\label{7quo}
(M/N): Heymod_H^2\longrightarrow Heymod_H\qquad P\mapsto \{\mbox{$(f,g)\in Heymod_H^2(M,P)$ $\vert$ $f(x)=g(x)$ for all $x\in N$}\}
\end{equation}

\begin{lem}\label{L7.12} (a) For $N\subseteq M$ in $Heymod_H$, the association in \eqref{7quo} defines a functor
$(M/N): Heymod_H^2\longrightarrow Heymod_H$. 

\smallskip
(b) For each $P\in Heymod_H^2$, the involution $\sigma(P):(M/N)(P)\longrightarrow (M/N)(P)$ given by 
$(f,g)\mapsto (g,f)$ determines an involutive natural transformation of functors $\sigma:(M/N)\longrightarrow (M/N)$. 

\end{lem}

\begin{proof}
We consider a morphism $(f',g')\in Heymod_H^2(P,P')$ and some $(f,g)\in (M/N)(P)$.   By definition, the composition $(f',g')\circ (f,g)$ is given by $(f'\circ f+ g'\circ g,f'\circ g+ g'\circ f )$. Since $f|N=g|N$, it is clear that $(f'\circ f+ g'\circ g)|N=(f'\circ g+ g'\circ f )|N$. This proves (a). The result of (b) is also clear from the definitions. 
\end{proof}

\begin{thm}\label{P7.13}
Let $N\subseteq M$ in $Heymod_H$. 

\smallskip
(a) The cokernel pair of the inclusion $i:N\hookrightarrow M$ is given by the quotient 
of $M\times M$ over the equivalence relation
\begin{equation}\label{eq7.26}
(x,y)\sim (x',y')\quad \Leftrightarrow\quad f(x)\vee g(y)=f(x')\vee g(y'), \textrm{ }\forall\textrm{ }P\in Heymod_H^2,
\textrm{ }(f,g)\in (M/N)(P)
\end{equation}

\smallskip
(b) Set $Q:=Coker_p(i)$. Then, there is a canonical isomorphism of functors $(M/N)\circ \kappa_H \overset{\cong}{\longrightarrow}Heymod_H(Q,\_\_)$, where $\kappa_H$ is the canonical embedding $\kappa_H: Heymod_H\hookrightarrow Heymod_H^2$. 
\end{thm}

\begin{proof} Let $e_1$, $e_2:M\hookrightarrow M^2$ be the canonical inclusions.  From Proposition \ref{vP5.5} and from the definition in \eqref{cokerp}, we know that the cokernel pair of $i:N\hookrightarrow M$
is given by taking the quotient of $M\times M$ over the equivalence relation $(x,y)\sim (x',y')$ if $t(x,y)=t(x',y')$
for every $t:M\times M\longrightarrow P$ in $Heymod_H$ such that $t\circ (e_1\circ i)=t\circ (e_2\circ i)$. 

\smallskip
Each $t:M\times M\longrightarrow P$ corresponds to an ordered pair $(f,g)$ of morphisms from $M$ to $P$ such that
$t(x,y)=f(x)\vee g(y)$ for each $(x,y)\in M\times M$. Further $t:M\times M\longrightarrow P$
satisfies  $t\circ (e_1\circ i)=t\circ (e_2\circ i)$ if and only if $f(n)=g(n)$ for each $n\in N$. The result of (a) is now clear. 

\smallskip
As such, any morphism in $Heymod_H(Q,P)$ corresponds to a morphism $t:M\times M\longrightarrow P$ in $Heymod_H$
such that $t(x,y)=t(x',y')$ whenever $(x,y)\sim (x',y')$. Suppose that $t$ is given by the ordered pair 
$(f_1,f_2)$ of morphisms $M\longrightarrow P$. If $n\in N$, we notice that $(n,0)\sim (0,n)$ for the equivalence relation in \eqref{eq7.26}. 
Then, $f_1(n)=t(n,0)=t(0,n)=f_2(n)$, i.e., $(f_1,f_2)\in ((M/N)\circ \kappa_H)(P)$. 

\smallskip
Conversely, we consider $(g_1,g_2)\in  ((M/N)\circ \kappa_H)(P)$, i.e., morphisms $g_1,g_2:M\longrightarrow P$
such that $g_1|N=g_2|N$. Then, if $(x,y)\sim (x',y')$ as in \eqref{eq7.26}, we must have $g_1(x)\vee g_2(y)=
g_1(x')\vee g_2(y')$. Then, the morphism $t:M\times M\longrightarrow P$ given by the ordered pair $(g_1,g_2)$
satisfies $t(x,y)=t(x',y')$ whenever $(x,y)\sim (x',y')$. Hence, $(g_1,g_2)$ induces a morphism $Q\longrightarrow P$
in $Heymod_H$. This proves (b). 

\end{proof}

We conclude this section by explaining when the involution $\sigma$ in Lemma \ref{L7.12} is an identity. This will require an application of Proposition \ref{P4.4}, which is our analogue of Hahn-Banach theorem. As such, we will have to assume that
$H$ is a finite Boolean algebra.

\begin{thm}\label{P7.14} Let $H$ be a finite Boolean algebra and consider $N\subseteq M$ in $Heymod_H$. Then, 
the involutive natural transformation of functors  $\sigma : (M/N)\longrightarrow (M/N)$ described in Lemma \ref{L7.12}
is the identity if and only if $N=M$. 

\end{thm}

\begin{proof} If $\sigma : (M/N)\longrightarrow (M/N)$ is the identity, it follows in particular that $\sigma(H):
(M/N)(H)\longrightarrow (M/N)(H)$ is an identity. This means that $\phi_1=\phi_2$ for any maps 
$\phi_1$, $\phi_2\in M^\bigstar$ such that $\phi_1|N=\phi_2|N$. Applying Proposition \ref{P4.4}, we get $N=M$. 

\end{proof}

\section{The Eilenberg-Moore category $\mathbf{Heymod_H^{\mathfrak s}}$}

In Proposition \ref{P6.2}, we have observed that the endofunctor $\bigperp :Heymod_H\longrightarrow Heymod_H$ defined by taking any object $M\in Heymod_H$ to $M^2$ and any morphism $f:M\longrightarrow N$ to $(f,f):M^2
\longrightarrow N^2$ determines a comonad on $Heymod_H$.

\begin{defn}\label{D8.1} (see, for instance, \cite[p 189]{Borceux})
Let $\mathcal C$ be a category along with a triple $(\bigperp,\delta,\epsilon)$ determining a comonad on $\mathcal C$. A coalgebra
over the comonad $(\bigperp,\delta,\epsilon)$ is a pair $(A,\xi)$ consisting of an object $A\in\mathcal C$ and a morphism
$\xi:A\longrightarrow \bigperp A$ such that
\begin{equation}\label{eq8.1}
\epsilon(A)\circ \xi=id_A\qquad \bigperp(\xi)\circ \xi=\delta(A)\circ \xi : A\longrightarrow \bigperp\bigperp A
\end{equation}
The category of such coalgebras   is said to be the Eilenberg-Moore category of the comonad  $(\bigperp,\delta,\epsilon)$. 
\end{defn}

For the category $Heymod_H$ and the comonad $\bigperp$, a coalgebra consists of some $M\in Heymod_H$ and $\xi:
M\longrightarrow \bigperp M$ satisfying the conditions in \eqref{eq8.1}. In particular, $\epsilon(M)\circ \xi=id$ and hence
$\xi$ is of the form $x\mapsto (x,\sigma(x))\in M\times M$ for each $x\in M$. It may be verified in a manner similar to \cite[Proposition 3.12]{one}
that $\sigma:M\longrightarrow M$ is actually an involution and that the Eilenberg-Moore category of the comonad $\bigperp$
on $Heymod_H$ is equivalent to the category of Heyting modules equipped with an ($H$-linear) involution.

\smallskip In this section, we will study
this Eilenberg-Moore category, which we call $Heymod_H^{\mathfrak s}$ by extending the terminology from \cite[$\S$ 5]{one}.  A morphism $f:(M,\sigma)\longrightarrow (M',\sigma')$ in $Heymod_H^{\mathfrak s}$ is a morphism $f:M\longrightarrow M'$ in 
$Heymod_H$ that commutes with the involutions, i.e., $f\circ \sigma=\sigma'\circ f$. 

\smallskip
For any Heyting module $N$, we notice that $Heymod_H(H,N)=N$, i.e., each $n\in N$ corresponds to the morphism $f_n:H\longrightarrow N$ which takes $c\in H$ to $c\wedge n$. We now consider the functor:
\begin{equation}\label{yon}
y_H:=Heymod_H^2(H,\_\_):Heymod_H^2\longrightarrow Heymod_H\qquad N\mapsto Heymod_H^2(H,N)=N\times N
\end{equation} Further, the Heyting module $N\times N$ is equipped with an obvious involution $\tau_N$ that takes $(n_1,n_2)\in N\times N$
to $(n_2,n_1)$. This involution may also be obtained by considering the morphism $(0,id):N\longrightarrow N$ in
$Heymod_H^2$ and applying the functor $y_H$ in \eqref{yon}. Hence, $y_H$ may be rewritten as a functor
$y_H:=Heymod_H^2(H,\_\_):Heymod_H^2\longrightarrow Heymod_H^{\mathfrak s}$.

\smallskip
In general, if $\tilde g=(g_1,g_2):N\longrightarrow N'$ is a morphism in $Heymod_H^2$, the corresponding morphism
$y_H(\tilde g):N\times N\longrightarrow N'\times N'$ is given by
\begin{equation}\label{eq8.3}
\begin{array}{c}
(n_1,n_2)\in N\times N\mapsto (f_{n_1},f_{n_2})\in Heymod_H^2(H,N)\mapsto\hspace{2.5in}\\ 
\hspace{0.1in}(g_1,g_2)\circ  (f_{n_1},f_{n_2})\in Heymod_H^2(H,N')\mapsto (g_1(n_1)\vee g_2(n_2),g_1(n_2)\vee g_2(n_1))\in N'\times N' \\
\end{array}
\end{equation} We also notice that the morphism in \eqref{eq8.3} is compatible with the respective involutions
on $N\times N$ and $N'\times N'$.   Comparing \eqref{eq8.3} and the definition in \eqref{kerp1}, it  follows that the kernel pair of a morphism  $\tilde g=(g_1,g_2):N\longrightarrow N'$  in $Heymod_H^2$ is given by
\begin{equation}\label{kerp8}
Ker_p(\tilde{g}):=\{\mbox{$(n_1,n_2)\in N\times N$ $\vert$ $g_1(n_1)\vee g_2(n_2)=g_1(n_2)\vee g_2(n_1)$ }\}=y_H(\tilde g)^{-1}(\Delta_{N'})
\end{equation} where $\Delta_{N'}\subseteq N'\times N'$ is the diagonal. Using \eqref{eq8.3}, the definition in \eqref{im2}
may also be recast as 
\begin{equation}\label{im8}
I(\tilde g):=\{\mbox{$(g_1(n_1)\vee g_2(n_2),g_1(n_2)\vee g_2(n_1))$ $\vert$ $n_1$, $n_2\in N$}\}=Range(y_H(\tilde g))
\end{equation} It may be verified in a way analogous to \cite[Lemma 5.1]{one} that $y_H:Heymod_H^2\longrightarrow Heymod_H^{\mathfrak s}$ embeds $Heymod_H^2$ as a full subcategory of $Heymod_H^{\mathfrak s}$. The result of Lemma \ref{L6.6} may now be adapted to the category $Heymod_H^{\mathfrak s}$ as follows.

\begin{thm}\label{P8.2} Let $\begin{CD} L@>\tilde f=(f_1,f_2)>> M@>\tilde g=(g_1,g_2)>> N\end{CD}$ be morphisms
in $Heymod_H^2$. Then, the following are equivalent

\smallskip
(a) $
I(\tilde f)\subseteq Ker_p(\tilde g)$ 

\smallskip
(b) $g_1\circ f_1+g_2\circ f_2=g_1\circ f_2+g_2\circ f_1$

\smallskip
(c) $Range(y_H(\tilde f))\subseteq y_H(\tilde g)^{-1}(\Delta_{N})$.

\smallskip
(d) $Range(y_H(\tilde g\circ \tilde f))\subseteq \Delta_{N}$.

\end{thm}

\begin{proof}
The fact that (a) $\Leftrightarrow$ (b) already follows from Lemma \ref{L6.6}. Using \eqref{kerp8} and \eqref{im8}, it is clear that
(c) $\Leftrightarrow$ (a). Using \eqref{kerp8}, we also see that $Range(y_H(\tilde g\circ \tilde f))\subseteq \Delta_{N}$ is equivalent
to the saying that $Ker_p(\tilde g\circ \tilde f)=L\times L$. Since $\tilde g\circ \tilde f=(g_1f_1+g_2f_2,g_1f_2+g_2f_1)$, it follows that  $Ker_p(\tilde g\circ \tilde f)=L\times L$ is further equivalent to the condition that 
\begin{equation}\label{eq8.6}
(g_1f_1+g_2f_2)(l_1)\vee (g_1f_2+g_2f_1)(l_2)=(g_1f_1+g_2f_2)(l_2)\vee (g_1f_2+g_2f_1)(l_1)
\end{equation} for every $l_1$, $l_2\in L$. Hence, (b) $\Rightarrow$ (d). Putting $l_2=0$ in \eqref{eq8.6}, we get
(d) $\Rightarrow$ (b). 
\end{proof}

\begin{thm}\label{P8.3} Let $\tilde g=(g_1,g_2):M\longrightarrow N$ be a morphism in $Heymod_H^2$. Then:

\smallskip
(a) $\tilde g$ is a monomorphism in $Heymod_H^2$ if and only if $y_H(\tilde g)$ is injective.

 \smallskip
(b) If $y_H(\tilde g)$ is surjective, then $\tilde g$ is an epimorphism in $Heymod_H^2$. 

\smallskip
(c) Let $H$ be a finite Boolean algebra. Then, $y_H(\tilde g)$ is surjective if and only if $\tilde g$ is an epimorphism in $Heymod_H^2$. 

\end{thm}

\begin{proof} From Proposition \ref{P6.8}(a), we know that $\tilde g$ is a monomorphism
if and only if the induced map $M^2\longrightarrow N^2$ given by $(m_1,m_2)\mapsto (g_1(m_1)\vee g_2(m_2),g_1(m_2)\vee g_2(m_1))$ is injective. From \eqref{eq8.3}, this is equivalent to $y_H(\tilde g)$ being injective. This proves (a). 

\smallskip
From Proposition \ref{P7.9}, we know that $y_H(\tilde g)$ being surjective, i.e.,  $I(\tilde g)=Range(y_H(\tilde g))=N\times N$,  is equivalent
to the sequence $M\overset{\tilde g}{\longrightarrow}N\longrightarrow 0$ being strictly exact at $N$. In particular, Proposition
\ref{P7.9} also says that this makes $\tilde g$ an epimorphism in $Heymod_H^2$. This proves (b). Additionally, if
$H$ is a finite Boolean algebra, it follows from Proposition \ref{P7.10} that the sequence $M\overset{\tilde g}{\longrightarrow}N\longrightarrow 0$ being strictly exact at $N$ is equivalent to $\tilde g$ being an epimorphism in $Heymod_H^2$. This proves (c).

\end{proof}

Let us denote by $\mathfrak s$ the squaring functor
\begin{equation}
\begin{CD}\mathfrak s:Heymod_H@>\kappa_H>>  Heymod_H^2@>y_H>> Heymod_H^{\mathfrak s}\end{CD}
\end{equation} Then, $\mathfrak s(N)=(N\times N,\tau_N)$ for any $N\in Heymod_H$, where $\tau_N(n_1,n_2)=(n_2,n_1)$  for each $(n_1,n_2)\in N\times N$.  For a morphism $f:N\longrightarrow N'$ in $Heymod_H$,
the induced morphism $\mathfrak s(f)$ is given by $(f,f):(N\times N,\tau_N)\longrightarrow (N'\times N',\tau_{N'})$. We consider a morphism
$\phi:(M,\sigma)\longrightarrow \mathfrak s(N)=(N\times N,\tau_N)$ in $Heymod_H^{\mathfrak s}$.  If $\phi: M\longrightarrow N\times N$ is given by $(f,g)$, the following diagram must be commutative
\begin{equation}\label{eq8.8}
\begin{CD}
M @>\phi=(f,g)>> N^2 \\
@V\sigma VV @V\tau_NVV \\
M @>\phi=(f,g)>> N^2 \\
\end{CD}
\end{equation} From \eqref{eq8.8}, we obtain $(g(m),f(m))=(f(\sigma(m)),g(\sigma(m)))$ for each $m\in M$ and hence
$g(m)=f(\sigma(m))$. Hence, given the object $(M,\sigma)\in Heymod_H^{\mathfrak s}$, the morphism  $\phi=(f,g):(M,\sigma)\longrightarrow \mathfrak s(N)=(N\times N,\tau_N)$ is determined completely by the morphism $f:M\longrightarrow N$ in $Heymod_H$. This gives us an adjunction
of functors
\begin{equation}\label{adj8}
Heymod_H^{\mathfrak s}((M,\sigma),\mathfrak s(N))\cong Heymod_H(\mathfrak f(M,\sigma),N)
\end{equation}  Here $\mathfrak f:Heymod_H^{\mathfrak s}\longrightarrow Heymod_H$ is the forgetful functor. We observe
that \eqref{adj8} is actually an isomorphism of Heyting modules. 

\smallskip
Suppose now that $N$ is a Heyting module and $M\subseteq N$ a Heyting submodule. From Proposition \ref{P7.13}, we know that the cokernel pair $Q=Coker_p(i)$ of the inclusion $i:M\hookrightarrow N$ is given by the quotient 
of $N\times N$ over the equivalence relation
\begin{equation}\label{eq8.10}
(x,y)\sim (x',y')\quad \Leftrightarrow\quad f(x)\vee g(y)=f(x')\vee g(y'), \textrm{ }\forall\textrm{ }P\in Heymod_H^2,
\textrm{ }(f,g)\in (N/M)(P)
\end{equation} From the definition in \eqref{7quo}, it is clear that $(f,g)\in (N/M)(P)$ if and only if $(g,f)\in (N/M)(P)$. Then, if $(x,y)\sim (x',y')$
according to the eqivalence relation in \ref{eq8.10},  we must also have $g(x)\vee f(y)=g(x')\vee f(y'), \textrm{ }\forall\textrm{ }P\in Heymod_H^2,
\textrm{ }(f,g)\in (N/M)(P)$. In other words, if $(x,y)\sim (x',y')$ in $N\times N$, we must also have $(y,x)\sim (y',x')$. This means that the cokernel pair $Q$ is equipped with a canonical  involution $\sigma$,
making it an object of $Heymod_H^{\mathfrak s}$. 

\smallskip
The adjunction in \eqref{adj8} now gives us an isomorphism
\begin{equation}\label{eq8.11}
Heymod_H^{\mathfrak s}((Q,\sigma),\mathfrak s(P))\cong Heymod_H(\mathfrak f(Q,\sigma),P)
\end{equation} for any object $P\in Heymod_H$. Since $\mathfrak f(Q,\sigma)$ is simply $Q$ considered again as an object of 
$Heymod_H$, it follows from Proposition \ref{P7.13}(b) that $(N/M)(\kappa_H(P)) \overset{\cong}{\longrightarrow}Heymod_H(\mathfrak f(Q,\sigma),P)$. Combining with \eqref{eq8.11}, we get
\begin{equation}\label{eq8.12}
Heymod_H^{\mathfrak s}((Q,\sigma),y_H(\kappa_H(P)))=Heymod_H^{\mathfrak s}((Q,\sigma),\mathfrak s(P))\cong Heymod_H(\mathfrak f(Q,\sigma),P)=(N/M)(\kappa_H(P))
\end{equation} It follows from \eqref{eq8.12} that the functors $Heymod_H^{\mathfrak s}((Q,\sigma),y_H(\_\_)),
(N/M)(\_\_):Heymod_H^2\longrightarrow Heymod_H$ coincide when restricted to $\kappa_H:Heymod_H
\hookrightarrow Heymod_H^2$. It may be easily verified that the isomorphisms in \eqref{eq8.12} are well-behaved
with respect to morphisms in $Heymod_H^2$ and we obtain
\begin{equation}\label{q8.13}
Heymod_H^{\mathfrak s}((Q,\sigma),y_H(P))\cong 
(N/M)(P)
\end{equation} for every $P\in Heymod_H^2$. 

\smallskip

For an object $(M,\sigma)$ in $Heymod_H^{\mathfrak s}$, we will always denote by $M^\sigma$ the collection
of fixed points of the involution $\sigma$. Since $\sigma$ is $H$-linear, it is clear that $M^\sigma\in Heymod_H$. 

\smallskip We now consider  a sequence $\begin{CD} L@>\tilde f=(f_1,f_2)>> M@>\tilde g=(g_1,g_2)>> N
\end{CD} $ of morphisms in $Heymod_H^2$ and the corresponding sequence $\begin{CD} L^2@>y_H(\tilde f)>> M^2@>
y_H(\tilde g)>> N^2
\end{CD} $ in 
$Heymod_H^{\mathfrak s}$.  We know from Proposition \ref{P8.2} that $I(\tilde f)\subseteq Ker_p(\tilde g)$
if and only if $Range(y_H(\tilde f))\subseteq y_H(\tilde g)^{-1}(\Delta_N)$. Expressing the diagonal $\Delta_N\subseteq 
N\times N$ as the collection of fixed points of the involution $\tau_N:N\times N\longrightarrow N\times N$ on $y_H(N)$, we 
can rewrite this condition as $Range(y_H(\tilde f))\subseteq y_H(\tilde g)^{-1}(y_H(N)^{\tau_N})$.

\smallskip
This motivates the idea that a composition of morphisms $\begin{CD}(L,\sigma_L)@>f>> (M,\sigma_M)
@>g>> (N,\sigma_N)\end{CD}$ in $Heymod_H^{\mathfrak s}$ should be treated as ``zero'' if $Range(f)\subseteq
g^{-1}(N^{\sigma_N})$.

\begin{defn}\label{D8.4} A sequence $\begin{CD} (L,\sigma_L)@>f>> (M,\sigma_M)
@>g>> (N,\sigma_N)
\end{CD} $ in $Heymod_H^{\mathfrak s}$ is strictly exact at $M$ if $Range(f)+M^{\sigma_M}=g^{-1}(N^{\sigma_N})$. 

\end{defn}

\begin{thm}\label{P8.5} Let $\begin{CD} L@>\tilde f=(f_1,f_2)>> M@>\tilde g=(g_1,g_2)>> N
\end{CD} $ be a sequence of morphisms in  $Heymod_H^2$. Then, the following are equivalent:

\smallskip
(a) In $Heymod_H^2$, the sequence $\begin{CD} L@>\tilde f>> M@>\tilde g>> N
\end{CD} $ is strictly exact at $M$. 

\smallskip
(b) In $Heymod_H^{\mathfrak s}$, the sequence $\begin{CD}y_H(L)@>y_H(\tilde f)>>y_H(M)@>y_H(\tilde g)>>y_H(N)\end{CD}$ is strictly exact at $y_H(M)$. 

\end{thm}

\begin{proof}  From \eqref{yon}, we know that $y_H(L)$, $y_H(M)$ and $y_H(N)$ are given respectively by
$(L\times L,\tau_L)$, $(M\times M,\tau_M)$, $(N\times N,\tau_N)$, each equipped with the canonical involution that swaps the two components. 
From Definition \ref{D6.7},  $\begin{CD} L@>\tilde f=(f_1,f_2)>> M@>\tilde g=(g_1,g_2)>> N
\end{CD} $ is strictly exact at $M$ if and only if $I(\tilde f)+\Delta_M=Ker_p(\tilde g)$. It is clear  that the diagonal $\Delta_M=(M\times M)^{\tau_M}=y_H(M)^{\tau_M}$. 

\smallskip From \eqref{kerp8}, we see that $Ker_p(\tilde g)=y_H(\tilde g)^{-1}(\Delta_N)=y_H(\tilde g)^{-1}((N\times N)^{\tau_N})=y_H(\tilde g)^{-1}(y_H(N)^{\tau_N})$. On the other hand, \eqref{im8} gives us $I(\tilde f)=Range(y_H(\tilde f))$. The result is now clear from Definition \ref{D8.4}.
\end{proof}

We now return to the adjunction
\begin{equation}\label{adj8'}
Heymod_H(\mathfrak f(M,\sigma_M),N)=Heymod_H^{\mathfrak s}((M,\sigma_M),\mathfrak s(N))
\end{equation} explained in \eqref{adj8}, where $(M,\sigma_M)\in Heymod_H^{\mathfrak s}$ and $N\in Heymod_H$. By definition,
$(\mathfrak s\circ \mathfrak f)(M,\sigma_M)=(M^2,\tau_M)$, where $\tau_M:M\times M\longrightarrow M\times M$ is the involution
that interchanges the components. The unit map of this adjunction $1_{Heymod_H^{\mathfrak s}}\longrightarrow (\mathfrak s\circ \mathfrak f)$ corresponds to the identity map $\mathfrak f(M,\sigma_M)=M\longrightarrow M=\mathfrak f(M,\sigma_M)$ in $Heymod_H$. 

\smallskip
Using the explicit description of this adjunction in \eqref{eq8.8}, we conclude that the morphism $\eta{(M,\sigma_M)}:(M,\sigma_M)\longrightarrow 
(\mathfrak s\circ \mathfrak f)(M,\sigma_M)=(M^2,\tau_M)$ in $Heymod_H^{\mathfrak s}$ given by the unit $1_{Heymod_H^{\mathfrak s}}\longrightarrow (\mathfrak s\circ \mathfrak f)$ corresponds to $(1_M,\sigma_M):M\longrightarrow M^2$. Further, if we put $\mathfrak T=\mathfrak s\circ \mathfrak f$, there is a retraction
\begin{equation}\label{retract}
\xi(M,\sigma_M) : \mathfrak T(M,\sigma_M)\longrightarrow (M,\sigma_M) \qquad (x,y)\mapsto x\vee \sigma_M(y)
\end{equation} satisfying $\xi(M,\sigma_M)\circ \eta(M,\sigma_M)=id$.

\begin{thm}\label{P8.6}
For a morphism $f:(L,\sigma_L)\longrightarrow (M,\sigma_M)$ in $Heymod_H^{\mathfrak s}$, the following are equivalent:

\smallskip
(a) $f$ is a monomorphism in $Heymod_H^{\mathfrak s}$.

\smallskip
(b) $f:L\longrightarrow M$ is injective. 

\smallskip
(c) The sequence $\begin{CD} 0@>>>\mathfrak T(L)@>\mathfrak T(f)>> \mathfrak T(M)\end{CD}$ is strictly exact at $\mathfrak T(L)$. 
\end{thm}

\begin{proof}
It is clear that (b) $\Rightarrow$ (a). For any element $x\in L$, there exists a canonical morphism $\xi_x:(H\times H,\tau_H)\longrightarrow (L,\sigma_L)$ given by $\xi_x(a,b)=(a\wedge x)\vee (b\wedge \sigma_L(x))$. Then, $f(x)=f(y)$ for $x$, $y\in L$ gives
$f\circ \xi_x=f\circ \xi_y$.  If $f$ is a monomorphism, then $\xi_x=\xi_y$ and hence $x=\xi_x(1,0)=\xi_y(1,0)=y$. Hence, (a) $\Rightarrow$ (b). 

\smallskip
By definition, $\mathfrak T=y_H\circ \kappa_H\circ \mathfrak f$. Hence, using Proposition \ref{P8.5}, we see that the sequence $\begin{CD} 0@>>>\mathfrak T(L)@>\mathfrak T(f)>> \mathfrak T(M)\end{CD}$ being strictly exact at $\mathfrak T(L)$ is equivalent to the sequence $\begin{CD} 0@>>>\kappa_H(\mathfrak f(L))@>\kappa_H(\mathfrak f(f))>> \kappa_H(\mathfrak f(M))\end{CD}$ being strictly exact at $\kappa_H(\mathfrak f(L))$. Applying Proposition \ref{P6.8}(c), the latter is equivalent to the statement that
\begin{equation}
f(x)=f(y) \qquad \Leftrightarrow \qquad x=y
\end{equation} In other words, $f$ is injective. This proves the result. 

\end{proof}

\begin{thm}\label{P8.7}
For a morphism $f:(L,\sigma_L)\longrightarrow (M,\sigma_M)$ in $Heymod_H^{\mathfrak s}$, the following are equivalent:

\smallskip
(a) $f:L\longrightarrow M$ is surjective. 

\smallskip
(b) The sequence $\begin{CD}\mathfrak T(L)@>\mathfrak T(f)>> \mathfrak T(M)@>>>0\end{CD}$ is strictly exact at $\mathfrak T(M)$. 

\smallskip
In particular, either of these conditions implies that $f$ is an epimorphism in $Heymod_H^{\mathfrak s}$. 
\end{thm}

\begin{proof}
Again since $\mathfrak T=y_H\circ \kappa_H\circ \mathfrak f$, it follows from Proposition \ref{P8.5} that (b) is equivalent
to the sequence $\begin{CD} \kappa_H(\mathfrak f(L))@>\kappa_H(\mathfrak f(f))>> \kappa_H(\mathfrak f(M))@>>>0\end{CD}$ being strictly exact at $\kappa_H(\mathfrak f(M))$. Applying Proposition \ref{P7.9}(a), the latter is equivalent to the statement that $\{\mbox{$f(x)$ $\vert$ $f(y)=0$, 
$x$, $y\in L$}\}=M$. Hence, (a) $\Leftrightarrow$ (b).
\end{proof}

As in other sections, the best results for epimorphisms are obtained when $H$ is a finite Boolean algebra.

\begin{thm}\label{P8.8}
Let $H$ be a finite Boolean algebra. For a morphism $f:(L,\sigma_L)\longrightarrow (M,\sigma_M)$ in $Heymod_H^{\mathfrak s}$, the following are equivalent:

\smallskip
(a) $f$ is an epimorphism in $Heymod_H^{\mathfrak s}$.

\smallskip
(b) $f:L\longrightarrow M$ is surjective. 

\smallskip
(c) The sequence $\begin{CD}\mathfrak T(L)@>\mathfrak T(f)>> \mathfrak T(M)@>>>0\end{CD}$ is strictly exact at $\mathfrak T(M)$. 
\end{thm}

\begin{proof}
From Proposition \ref{P8.7}, we know that (b) $\Leftrightarrow$ (c) $\Rightarrow$ (a) for any finite Heyting algebra $H$. We now suppose (a), i.e., 
 $f$ is an epimorphism in $Heymod_H^{\mathfrak s}$.  The range $N:=Range(f)$ of $f$ is a Heyting submodule of $M$. If $N\ne M$, it follows from 
 Proposition \ref{P4.4} that we can choose morphisms $\phi_1\ne \phi_2:M=\mathfrak f(M,\sigma_M)\longrightarrow H$ in $Heymod_H$ such that $\phi_1\circ f=\phi_2\circ f$. The adjunction in \eqref{adj8} then gives $\tilde\phi_1\ne \tilde\phi_2\in Heymod_H^{\mathfrak s}(
 (M,\sigma_M),(H\times H,\tau_H))$ corresponding respectively to $\phi_1$
 and $\phi_2$.  For any $x\in L$, we now have
 \begin{equation*}
 (\tilde\phi_1\circ f)(x)=(\phi_1(f(x)),\phi_1\sigma_M(f(x)))=(\phi_1(f(x)),(\phi_1\circ f)(\sigma_L(x)))=(\phi_2(f(x)),(\phi_2\circ f)(\sigma_L(x)))= (\tilde\phi_2\circ f)(x)
 \end{equation*} Since $f$ is an epimorphism in $Heymod_H^{\mathfrak s}$, this gives $\tilde\phi_1=\tilde\phi_2$ and hence $\phi_1=\phi_2$,
 which is a contradiction. 
\end{proof}

\section{$Heymod_H^{\mathfrak s}$ as a semiexact category}

For a finite Heyting algebra $H$, we will show in this section that $Heymod_H^{\mathfrak s}$ is a semiexact category
in the sense of Grandis \cite{Grand0}, \cite[$\S$ 1.3.3]{Grandis}. For this, we will first recall several notions from \cite[Chapter 1]{Grandis}.

\smallskip
Let $\mathcal C$ be a category. A collection $\mathscr N$ of morphisms of $\mathcal C$ is said to be an ideal
if $f\in \mathscr N$ implies that $g\circ f\circ h\in \mathscr N$ for all morphisms  $g,h$ in $\mathcal C$ such that the composition
$g\circ f\circ h$ is legitimate. Further, $\mathscr N$ is said to be a closed ideal if every morphism in $\mathscr N$
factorizes through some identity morphism also in $\mathscr N$. 

\smallskip
An $N$-category is a pair $(\mathcal C,\mathscr N)$ consisting of a category $\mathcal C$ and an ideal $\mathscr N$ of morphisms
of $\mathcal C$. The ideal $\mathscr N$ is referred to as the ideal of null morphisms of $\mathcal C$. An object $X$ of $\mathcal C$
is said to be null if the identity morphism $id_X\in \mathscr N$. A functor $F:(\mathcal C,\mathscr N)\longrightarrow 
(\mathcal C',\mathscr N')$ of $N$-categories is a functor $F:\mathcal C\longrightarrow \mathcal C'$ that preserves null morphisms. 

\begin{defn}\label{Dt9.1} Let $(\mathcal C,\mathscr N)$ be an $N$-category and let $f:A\longrightarrow B$ be a morphism
in $\mathcal C$. 

\smallskip A morphism $k:K\longrightarrow A$ is said to be a kernel for $f$ if it satisfies the following two conditions:

\smallskip
(1) $f\circ k\in \mathscr N$.

(2) If $h$ is a morphism in $\mathcal C$ such that $f\circ h\in \mathscr N$, then $h$ factorizes uniquely through $k$. 

\smallskip
A morphism $c:B\longrightarrow C$ is said to be a cokernel for $f$ if it satisfies the following two conditions:

\smallskip
(1) $c\circ f\in \mathscr N$.

(2) If $h$ is a morphism in $\mathcal C$ such that $h\circ f\in \mathscr N$, then $h$ factorizes uniquely through $c$.

\end{defn}

\begin{defn}\label{Dt9.2} A semiexact category  is an $N$-category $(\mathcal C,\mathscr N)$ such that

\smallskip
(a) $\mathscr N$ is a closed ideal.

\smallskip
(b) Every morphism in $\mathcal C$ has a kernel and a cokernel with respect to $\mathscr N$. 

\end{defn}

We now consider the category $Heymod_H^{\mathfrak s}$. We will say that a morphism $f:(L,\sigma_L)\longrightarrow (M,\sigma_M)$
in $Heymod_H^{\mathfrak s}$ is null if $f(x)=\sigma_M(f(x))$ for each $x\in L$. The collection of null morphisms
of $Heymod_H^{\mathfrak s}$ will be denoted by $\mathscr N$.

\begin{lem}\label{Lt9.3} The collection of null morphisms is a closed ideal in $Heymod_H^{\mathfrak s}$. 

\end{lem}

\begin{proof}
We consider $f:(L,\sigma_L)\longrightarrow (M,\sigma_M)$ in $\mathscr N$ and morphisms  $h:(L',\sigma_{L'})\longrightarrow (L,\sigma_L)$ and $g:(M,\sigma_M)\longrightarrow (M',\sigma_{M'})$
in $Heymod_H^{\mathfrak s}$. For $x'\in L'$, we have
\begin{equation*}gfh(x')=g(\sigma_M(fh(x')))=\sigma_{M'}(gfh(x'))
\end{equation*} This shows thatb $gfh\in \mathscr N$, i.e., $\mathscr 
N$ is an ideal.

\smallskip We consider the object $(M^{\sigma_M},id)\in Heymod_H^{\mathfrak s}$. It is clear that the identity on  $(M^{\sigma_M},id)$ is a null morphism in $Heymod_H^{\mathfrak s}$. Given 
 $f:(L,\sigma_L)\longrightarrow (M,\sigma_M)$ in $\mathscr N$, the condition $f=\sigma_M\circ f$ ensures that $f$ factorizes
 through $(M^{\sigma_M},id)$. This proves that $\mathscr N$ is a closed ideal. 
\end{proof}

\begin{thm}\label{Pt9.4}
Let $f:(L,\sigma_L)\longrightarrow (M,\sigma_M)$ be a morphism 
in $Heymod_H^{\mathfrak s}$. Then, the canonical morphism $(f^{-1}(M^{\sigma_M}),\sigma_L|f^{-1}(M^{\sigma_M}))
\longrightarrow (L,\sigma_L)$ is the kernel for $f$ with respect to $\mathscr N$. 
\end{thm}

\begin{proof}
First, we notice that $\sigma_L:L\longrightarrow L$ does restrict to an involution on $f^{-1}(M^{\sigma_M})$. Indeed, if $x\in f^{-1}(M^{\sigma_M})$, then $\sigma_Mf(\sigma_L(x))=\sigma_M^2f(x)=f(x)=\sigma_Mf(x)=f(\sigma_L(x))$, i.e., 
$\sigma_L(x)\in f^{-1}(M^{\sigma_M})$. Also, the composition $(f^{-1}(M^{\sigma_M}),\sigma_L|f^{-1}(M^{\sigma_M}))
\longrightarrow (L,\sigma_L)\overset{f}{\longrightarrow} (M,\sigma_M)$ factors through the null object $(M^{\sigma_M},id)$. 

\smallskip
By definition, a composition $(L',\sigma_{L'})\overset{h}{\longrightarrow} (L,\sigma_L)\overset{f}{\longrightarrow} (M,\sigma_M)$
is null if and only if $h(y)\in f^{-1}(M^{\sigma_M})$ for every $y\in L'$. Hence, any such null composition in $Heymod_H^{\mathfrak s}$  factors uniquely
through the  canonical morphism $(f^{-1}(M^{\sigma_M}),\sigma_L|f^{-1}(M^{\sigma_M}))
\longrightarrow (L,\sigma_L)$. This proves the result.
\end{proof}

Accordingly, the canonical morphism described in Proposition \ref{Pt9.4} will be written as the kernel $Ker(f)$ of 
 $f:(L,\sigma_L)\longrightarrow (M,\sigma_M)$. We now define an equivalence relation $\sim$ on $M$ as follows: 
\begin{equation}\label{eqf9.1}
x_1{\sim}x_2 \quad\Leftrightarrow\quad g(x_1)=g(x_2)\textrm{ }\forall \mbox{$g\in Heymod_H^{\mathfrak s}((M,\sigma_M),(N,\sigma_N))$
such that $g\circ f\in\mathscr N$}
\end{equation} It is easily seen that $M/\sim$ is a Heyting module. Further, if $x_1\sim x_2$ and $g\in Heymod_H^{\mathfrak s}((M,\sigma_M),(N,\sigma_N))$ is such that $g\circ f\in\mathscr N$, we see that $g(\sigma_M(x_1))=\sigma_Ng(x_1)=\sigma_Ng(x_2)
= g(\sigma_M(x_2))$. It follows from \eqref{eqf9.1} that $\sigma_M(x_1)\sim \sigma_M(x_2)$, i.e., the involution $\sigma_M$
descends to an involution on $M/\sim$ that we continue to denote by $\sigma_M$. 

\begin{thm}\label{Pt9.5}
Let $f:(L,\sigma_L)\longrightarrow (M,\sigma_M)$ be a morphism 
in $Heymod_H^{\mathfrak s}$. Then, the canonical morphism $(M,\sigma_M)\longrightarrow (M/\sim,\sigma_M)$ is the cokernel for $f$ with respect to $\mathscr N$. 
\end{thm}
\begin{proof}
For any $g\in Heymod_H^{\mathfrak s}((M,\sigma_M),(N,\sigma_N))$ such that $g\circ f\in\mathscr N$, we note that
$g\sigma_Mf(x)=\sigma_Ngf(x)=gf(x)$. It follows from \eqref{eqf9.1} that $f(x)\sim \sigma_Mf(x)$ and hence the composition 
$(L,\sigma_L)\overset{f}{\longrightarrow} (M,\sigma_M)\longrightarrow (M/\sim,\sigma_M)$ is null. The definition in \eqref{eqf9.1}
also shows that any $g\in Heymod_H^{\mathfrak s}((M,\sigma_M),(N,\sigma_N))$ such that $g\circ f\in\mathscr N$ must
factor through $M/\sim$. Since the involution on $M/\sim$ is induced by $\sigma_M$, it is clear that any such $g$ factors uniquely
through $(M,\sigma_M)\longrightarrow (M/\sim,\sigma_M)$  in $Heymod_H^{\mathfrak s}$. This proves the result. 
\end{proof}
 
Accordingly, the canonical morphism described in Proposition \ref{Pt9.5} will be written as the kernel $Coker(f)$ of 
 $f:(L,\sigma_L)\longrightarrow (M,\sigma_M)$. 

\begin{Thm}\label{Tt9.6} Let $H$ be a finite Heyting algebra. Then, $Heymod_H^{\mathfrak s}$ is a semiexact category.

\end{Thm}

\begin{proof}
This follows from the definition of a semiexact category and by applying Lemma \ref{Lt9.3}, Proposition \ref{Pt9.4}
and Proposition \ref{Pt9.5}. 
\end{proof}

\section{Finite Boolean algebras and semiexact homological categories}

We  begin this section by recalling some more general facts for semiexact categories. In a semiexact category $(\mathcal C,\mathscr N)$, a morphism of the form $Ker(f)\longrightarrow L$ corresponding to some morphism $f\in \mathcal C(L,M)$ is  referred to as a normal monomorphism (see \cite[$\S$ 1.3.3]{Grandis}). Similarly,  a  morphism of the form $M\longrightarrow Coker(f)$ corresponding to some $f\in\mathcal C(L,M)$ is referred to as a normal epimorphism. 

\smallskip
Every morphism $f:L\longrightarrow M$ in a semiexact category $(\mathcal C,\mathscr N)$ admits a unique and natural factorization of the form (see \cite[$\S$ 1.5.5]{Grandis})
\begin{equation}\label{fnormal1}
\begin{CD}
Ker(f) @>>> L@>f>> M@>>> Coker(f) \\
@. @VpVV @AmAA @. \\
@. Coker(Ker(f)) @>\tilde{f}>> Ker(Coker(f)) @.\\
\end{CD}
\end{equation} It is clear that $p:L\longrightarrow Coker(Ker(f))$ is a normal epimorphism 
and $m:Ker(Coker(f))\longrightarrow M$ is a normal monomorphism. 

\smallskip The morphism $f$ is said to be exact if $\tilde{f}$ as defined in \eqref{fnormal1}  is an isomorphism. A morphism $f$ in a semiexact
category $(\mathcal C,\mathscr N)$ is exact if and only if it can be factored as $k\circ h$, where $h$ is a normal epimorphism
and $k$ is a normal monomorphism. 

\smallskip
In this section, we will always assume that $H$ is a finite Boolean algebra. In particular, this assumption will allow us to use Proposition \ref{P4.3} and Proposition \ref{P4.31}. We recall from Section 4 that if $M$ is a Heyting module, then the dual $M^\bigstar$ denotes the collection
$Heymod_H(M,H)$ of all Heyting module morphisms from $M$ to $H$. 
Our purpose in this section is to show that $Heymod_H^{\mathfrak s}$ is actually a ``semiexact homological category'' in the sense
of \cite[$\S$ 1.3]{Grandis}, which we will do in a manner analogous to \cite[$\S$ 6]{one}. 

\begin{lem}\label{Lq10.1}  Let $H$ be a finite Boolean algebra.

\smallskip
(a) For any indexing set $I$, the product $H^I$
of copies of $H$ is an injective object in $Heymod_H$. 

\smallskip
(b) Every Heyting module $M$ can be embedded as a submodule of a product of copies of $H$.  
\end{lem}

\begin{proof}(a)  Let $M\in Heymod_H$ and let $N\subseteq M$ be a Heyting submodule. 
By definition, any morphism $\phi:N\longrightarrow H^I$ in $Heymod_H$ corresponds to a family of morphisms $\{\phi_i:N\longrightarrow
H\}_{i\in I}$. Since $H$ is a finite Boolean algebra, it follows from Proposition \ref{P4.3}
that the induced morphism $M^\bigstar\longrightarrow N^\bigstar$ is surjective, i.e., we can choose for each $\phi_i:N\longrightarrow
H$ a morphism $\psi_i:M\longrightarrow H$ extending $\phi_i$. The $\{\psi_i\}_{i\in I}$ combine to yield a morphism
$\psi:M\longrightarrow H^I$ extending $\phi$. 

\smallskip
(b) Given $M\in Heymod_H$, we define a morphism $i_M:M\longrightarrow H^{M^\bigstar}$ by setting 
$i_M(m)=\{\phi(m)\}_{\phi\in M^\bigstar}$. Using Proposition \ref{P4.31}, we see that $i_M$ is an injective map.  It now follows from Proposition \ref{P4.5} that 
$M$ is a Heyting submodule of $H^{M^\bigstar}$. 
\end{proof}

\begin{lem}\label{Lq10.2} Let $H$ be a finite Boolean algebra. Then, $(E,\sigma_E)$ is an injective object
in $Heymod_H^{\mathfrak s}$ if and only if $\mathfrak f(E,\sigma_E)=E$ is injective in $Heymod_H$. 
\end{lem}

\begin{proof}
We suppose that $E\in Heymod_H$ is injective. Let $i:(L,\sigma_L)\longrightarrow (M,\sigma_M)$ be a monomorphism
in $Heymod_H^{\mathfrak s}$ and let $f:(L,\sigma_L)\longrightarrow (E,\sigma_E)$ be a morphism in $Heymod_H^{\mathfrak s}$. By Proposition \ref{P8.6}, we know the underlying map $i:L\longrightarrow M$ is injective and hence a monomorphism in $Heymod_H$ (by Proposition \ref{P4.5}). Hence, there is a morphism $g:M\longrightarrow E$ in $Heymod_H$
such that $g\circ i=f$. Setting $h=g+\sigma_E\circ g\circ \sigma_M$, it may be verified easily that $h\circ \sigma_M=\sigma_E
\circ h$ and that $h\circ i=f$. This shows that $(E,\sigma_E)\in Heymod_H^{\mathfrak s}$ is injective. 

\smallskip
Conversely, suppose that  $(E,\sigma_E)$ is an injective object
in $Heymod_H^{\mathfrak s}$. Using Lemma \ref{Lq10.1}, we can find an embedding $i:E\longrightarrow H^I$ in $Heymod_H$
for some indexing set $I$. Then, the map $E\longrightarrow H^I\times H^I$ given by $x\mapsto (i(x),i(\sigma_E(x)))$ is injective
and it is easily verified that this gives  a morphism $u:(E,\sigma_E)\longrightarrow \mathfrak s(H^I)$  in $Heymod_H^{\mathfrak s}$.  Since $(E,\sigma_E)\in Heymod_H^{\mathfrak s}$ is injective, we have a retraction $v:\mathfrak s(H^I)\longrightarrow (E,\sigma_E)$
such that $v\circ u=id_{(E,\sigma_E)}$. It follows that $\mathfrak f(v):H^I\times H^I\longrightarrow E$ is a retraction
of the map $\mathfrak f(u):E\longrightarrow H^I\times H^I$ in $Heymod_H$. By Lemma \ref{Lq10.1}, we know that
$H^I\times H^I$ is injective and hence it follows that $E$ is injective in $Heymod_H$.
\end{proof}

\begin{lem}\label{Lq10.3} Let $f:(L,\sigma_L)\longrightarrow (M,\sigma_M)$ be a normal monomorphism 
in $Heymod_H^{\mathfrak s}$, corresponding to the kernel of $g:(M,\sigma_M)\longrightarrow (N,\sigma_N)$. Then, there 
is a morphism $h:(N,\sigma_N)\longrightarrow (E,\sigma_E)$ such that:

\smallskip 
(1)  $(E,\sigma_E)$ is injective in $Heymod_H^{\mathfrak s}$, 

(2) $h:N\longrightarrow E$ an injective map and

(3) $Ker(g)=((L,\sigma_L)\overset{f}{\longrightarrow}(M,\sigma_M))=Ker(h\circ g)$.
\end{lem}

\begin{proof}
We recall from Section 8 the morphism $\eta(N,\sigma_N):(N,\sigma_N)\longrightarrow (N^2,\tau_N)$ corresponding to the unit
of the adjunction between the functors $\mathfrak f:Heymod_H^{\mathfrak s}\longrightarrow Heymod_H$ and $\mathfrak s:Heymod_H\longrightarrow Heymod_H^{\mathfrak s}$. The map underlying $\eta(N,\sigma_N)$ is given by $(1_N,\sigma_N):
N\longrightarrow N^2$. Using the expression for the kernel of a morphism in $Heymod_H^{\mathfrak s}$ obtained
in Proposition \ref{Pt9.4}, we notice that 
\begin{equation}
Ker(\eta(N,\sigma_N)\circ g)=(\eta(N,\sigma_N)\circ g)^{-1}((N^2)^{\tau_N})=g^{-1}(N^{\sigma_N})=Ker(g)
\end{equation}  Applying Lemma \ref{Lq10.1}, we can find an embedding $i:N\longrightarrow E'$ into an injective $E'$ in $Heymod_H$. It is clear that $\mathfrak s(i):\mathfrak s(N)\longrightarrow \mathfrak s(E')$ satisfies $(\mathfrak s(i))^{-1}((E'^2)^{\tau_{E'}})
=(N^2)^{\tau_N}$. It follows that
\begin{equation}
Ker(\mathfrak s(i)\circ \eta(N,\sigma_N)\circ g)=(\mathfrak s(i)\circ \eta(N,\sigma_N)\circ g)^{-1}((E'^2)^{\tau_{E'}})=Ker(\eta(N,\sigma_N)\circ g)=Ker(g)
\end{equation} Finally, since $\mathfrak f (\mathfrak s(E'))=E'^2$ is injective in $Heymod_H$, it follows from 
Lemma \ref{Lq10.2} that $(E,\sigma_E):=\mathfrak s(E')$ is injective in $Heymod_H^{\mathfrak s}$. It is clear from the constructions that the map $N\longrightarrow E$ underlying $\mathfrak s(i)\circ \eta(N,\sigma_N)$ is injective. This proves the result.  
\end{proof}

\begin{thm}\label{Pq10.4}
Let $H$ be a finite Boolean algebra. Then, the normal monomorphisms in $Heymod_H^{\mathfrak s}$ are stable under composition.
\end{thm}

\begin{proof}
We consider normal monomorphisms $i:(L,\sigma_L)\longrightarrow (M,\sigma_M)$ and $j:(M,\sigma_M)\longrightarrow
(N,\sigma_N)$. Using Lemma \ref{Lq10.3}, we can find a morphism $f:(M,\sigma_M)\longrightarrow (E,\sigma_E)$
with $(E,\sigma_E)$ injective in $Heymod_H^{\mathfrak s}$ such that $(L,\sigma_L)=Ker(f)=(f^{-1}(E^{\sigma_E}),\sigma_L)$.  Then, there exists a morphism
$g:(N,\sigma_N)\longrightarrow (E,\sigma_E)$ such that $g\circ j=f$. Since $j$ is a normal monomorphism, we can write
$(M,\sigma_M)=Ker(h)$ for some morphism $h:(N,\sigma_N)\longrightarrow (P,\sigma_P)$. Then, we have
an induced morphism $(g,h):(N,\sigma_N)
\longrightarrow (E,\sigma_E)\times (P,\sigma_P)$ and we claim that $j\circ i=Ker(g,h)$.

\smallskip
By definition, we know that $Ker(g,h)=(g,h)^{-1}(E^{\sigma_E}\times P^{\sigma_P})$. As such, an element $n\in N$ lies
in $Ker(g,h)$ if and only if $g(n)\in E^{\sigma_E}$ and $h(n)\in P^{\sigma_P}$. Since $j=Ker(h)$, the fact that
$h(n)\in P^{\sigma_P}$ shows that $n=j(m)$ for some $m\in M$. Then, $f(m)=g(j(m))=g(n)\in E^{\sigma_E}$ and since
$i=Ker(f)$, we obtain $m=i(l)$ for some $l\in L$. This gives us $(j\circ i:(L,\sigma_L)\longrightarrow (N,\sigma_N))=Ker(g,h)$ and the result follows. 
\end{proof}

The next aim is to show that the normal epimorphisms in $Heymod_H^{\mathfrak s}$ are stable under composition.

\begin{lem}\label{Lq10.5} Let $(\mathcal C,\mathscr N)$ be a semiexact category. Then,   
$e:M\longrightarrow N$ is a normal epimorphism in $(\mathcal C,\mathscr N)$ if and only if $e$  is equivalent to the morphism $M\longrightarrow Coker(Ker(e))$. 
\end{lem}
\begin{proof}
Suppose that $e$ is a normal epimorphism.  Then, $e$ automatically factorizes as the composition of a normal epimorphism followed
by a normal monomorphism, i.e., $e$ is exact (see \cite[$\S$ 1.5.5]{Grandis}). Then, the morphism $\tilde e$ appearing
in the factorization of $e$ as in \eqref{fnormal1}:
\begin{equation}\label{fnormal11}
\begin{CD}
Ker(e) @>j>> M@>e>> N@>q>> Coker(e) \\
@. @VpVV @AmAA @. \\
@. Coker(j)= Coker(Ker(e)) @>\tilde{e}>> Ker(Coker(e))=Ker(q) @.\\
\end{CD}
\end{equation} must be an isomorphism. Further since  $q\circ e\in\mathscr N$ and $e$ is a normal epimorphism,
it follows from \cite[Lemma 1.5.3(g)]{Grandis} that $q\in \mathscr N$. From the definitions, it is evident that
the kernel of a null morphism must be the identity and hence $m=1_N:Ker(q)\longrightarrow N$. It is now clear
from \eqref{fnormal11} that $e:M\longrightarrow N$ is equivalent to $p:M\longrightarrow Coker(Ker(e)\overset{j}{\longrightarrow}M)$. The converse is obvious.
\end{proof}

From Theorem \ref{Tt9.6}, we know that $Heymod_H^{\mathfrak s}$ is a semiexact category. Applying Lemma \ref{Lq10.5} and  the definition of
the cokernel given in Proposition \ref{Pt9.5}, we see that a morphism $f:(M,\sigma_M)\longrightarrow (N,\sigma_N)$
in $Heymod_H^{\mathfrak s}$ is a normal epimorphism if and only if $f:M\longrightarrow N$ is surjective and 
\begin{equation}\label{cond62}
f(m)=f(m')\quad\Leftrightarrow \quad \mbox{$g(m)=g(m')$ for all $g\in Heymod_H^{\mathfrak s}((M,\sigma_M),(P,\sigma_P))$ s.t. $Ker(f)\subseteq Ker(g)$}
\end{equation} for any $m$, $m'\in M$. 

\begin{thm}
\label{Pq10.6}
Let $H$ be a finite Boolean algebra. Then, the normal epimorphisms in $Heymod_H^{\mathfrak s}$ are stable under composition.
\end{thm}

\begin{proof}
We consider normal epimorphisms $p:(L,\sigma_L)\longrightarrow (M,\sigma_M)$ and $q:(M,\sigma_M)\longrightarrow 
(N,\sigma_N)$.  Since $q$ is a normal epimorphism, it follows from the 
criterion in \eqref{cond62} that for $l$, $l'\in L$, we have $qp(l)=qp(l')$ if and only if $g(p(l))=g(p(l'))$ for each
$g\in  Heymod_H^{\mathfrak s}((M,\sigma_M),(Y,\sigma_Y))$ such that $Ker(g)\supseteq Ker(q)$. 

\smallskip
Since $p$ is a normal epimorphism, it follows from the explicit description of cokernels in Proposition \ref{Pt9.5} that
$p:L\longrightarrow M$ is surjective. As such, for a morphism $g\in  Heymod_H^{\mathfrak s}((M,\sigma_M),(Y,\sigma_Y))$, we have
$Ker(g)\supseteq Ker(q)$ $\Leftrightarrow$ $Ker(g\circ p)\supseteq Ker(q\circ p)$.  Setting $f:=g\circ p$, any such $g$
gives us a morphism $f\in Heymod_H^{\mathfrak s}((L,\sigma_L),(Y,\sigma_Y))$ such that $Ker(f)\supseteq Ker(q\circ p)$.

\smallskip
Conversely, suppose that we have a morphism  $f\in Heymod_H^{\mathfrak s}((L,\sigma_L),(Y,\sigma_Y))$ such that $Ker(f)\supseteq Ker(q\circ p)\supseteq Ker(p)$. Applying Lemma \ref{Lq10.5} to the normal epimorphism $p$, we know that $(M,\sigma_M)=Coker(Ker(p)\longrightarrow L)$ and hence
$f$ factors uniquely through some $g:(M,\sigma_M)\longrightarrow (Y,\sigma_Y)$ as $f=g\circ p$. 

\smallskip
Combining these facts, we have shown that $qp(l)=qp(l')$ for $l$, $l'\in L$ if and only if $f(l)=f(l')$ for any 
 $f\in Heymod_H^{\mathfrak s}((L,\sigma_L),(Y,\sigma_Y))$ such that $Ker(f)\supseteq Ker(q\circ p)$. Since $q$ and $p$ are both surjective,
 so is $q\circ p$. The result now follows from the criterion in \eqref{cond62}.
\end{proof}

\begin{defn}\label{Dq10.625} (see \cite[$\S$ 1.3.6]{Grandis}) Let $(\mathcal C,\mathscr N)$ be a semiexact category. Then, $(\mathcal C,\mathscr N)$ is said to be a homological category if it satisfies the following conditions:

\smallskip
(1) The normal monomorphisms in $(\mathcal C,\mathscr N)$  are stable under composition. 

\smallskip
(2) The normal epimorphisms in $(\mathcal C,\mathscr N)$ are stable under composition.

\smallskip
(3) Given a normal monomorphism $i:M\longrightarrow N$ and a normal epimorphism $q:N\longrightarrow Q$ in $(\mathcal C,\mathscr N)$ such that $Ker(q)\leq M$ in the lattice of subobjects of $M$, the composition $q\circ i$ is exact.

\end{defn}

\begin{lem}\label{Lq10.65}
Let $i:(M,\sigma_M)\longrightarrow (N,\sigma_N)$ be a normal monomorphism in $Heymod_H^{\mathfrak s}$. Then, for any
$f:(L,\sigma_L)\longrightarrow (M,\sigma_M)$ in $Heymod_H^{\mathfrak s}$, we have $Ker(i\circ f)=Ker(f)$. 
\end{lem}

\begin{proof}
Since $i:(M,\sigma_M)\longrightarrow (N,\sigma_N)$ is a normal monomorphism, we can choose some $g:(N,\sigma_N)
\longrightarrow (P,\sigma_P)$ such that $(M,\sigma_M)=Ker(g)$. From the definition of the kernel in Proposition \ref{Pt9.4}, it is clear that $i^{-1}(N^{\sigma_N})=M^{\sigma_M}$. Then, $Ker(i\circ f)=f^{-1}(i^{-1}(N^{\sigma_N}))=f^{-1}(M^{\sigma_M})=Ker(f)$. 
\end{proof}

\begin{thm}\label{Pq10.7}
Let $H$ be a finite Boolean algebra. Let $i:(M,\sigma_M)\longrightarrow (N,\sigma_N)$ (resp. $q:(N,\sigma_N)
\longrightarrow (Q,\sigma_Q)$) be a normal monomorphism (resp. a normal epimorphism) in $Heymod_H^{\mathfrak s}$. Suppose that 
$Ker(q)\subseteq (M,\sigma_M)$. Then, the composition $q\circ i$ is an exact morphism in $Heymod_H^{\mathfrak s}$.

\end{thm}

\begin{proof}
Since $i:(M,\sigma_M)\longrightarrow (N,\sigma_N)$ is a normal monomorphism, we may choose $f:(N,\sigma_N)\longrightarrow
(T,\sigma_T)$ such that $Ker(f)=(M,\sigma_M)$. By assumption, $Ker(q)\subseteq Ker(f)$. Since $q$ is a normal epimorphism, it follows from \eqref{cond62} that $q(n)=q(n')$ for $n$, $n'\in N$ implies that $f(n)=f(n')$. Accordingly, there is a morphism
$g:(Q,\sigma_Q)\longrightarrow (T,\sigma_T)$ such that $f=g\circ q$. We now set $(P,\sigma_P):=Ker(g)$. Since $g\circ q\circ i=
f\circ i\in \mathscr N$, there is a unique morphism $p:(M,\sigma_M)\longrightarrow Ker(g)=(P,\sigma_P)$ which makes the following diagram commutative
\begin{equation}
\begin{CD}
(M,\sigma_M) @>i>> (N,\sigma_N) \\
@VpVV @VqVV \\
(P,\sigma_P)@>j>> (Q,\sigma_Q)\\
\end{CD}
\end{equation}  Since  $(P,\sigma_P)=Ker(g)$, the morphism $j$ is a normal monomorphism. In order to show that $q\circ i=j\circ p$
is exact, it suffices therefore to show that $p$ is a normal epimorphism. In other words, we need to show that
$p$ coincides with the canonical morphism $p':(M,\sigma_M)\longrightarrow Coker(Ker(p)\longrightarrow (M,\sigma_M))$. Since $q$
is a normal epimorphism in $Heymod_H^{\mathfrak s}$, it follows from Proposition \ref{Pt9.5} that $q$ is surjective. We notice
that this implies that $p$ is surjective.

\smallskip
We now consider $m_1$, $m_2\in M$ such that $p'(m_1)=p'(m_2)$. Let $x:(N,\sigma_N)\longrightarrow (X,\sigma_X)$
be a morphism in $Heymod_H^{\mathfrak s}$ such that $Ker(q)\subseteq Ker(x)$. Then, $Ker(j\circ p)=Ker(q\circ i)\subseteq Ker(x\circ i)$ and it follows from Lemma \ref{Lq10.65} that $Ker(p)=Ker(j\circ p)\subseteq Ker(x\circ i)$. Since $p'$ is the canonical
morphism to the cokernel of $Ker(p)\longrightarrow M$, the fact that   $p'(m_1)=p'(m_2)$ now implies that $x( i(m_1))
=x(i(m_2))$.  Since $q$ is a normal epimorphism, it follows from \eqref{cond62} that $q(i(m_1))=q(i(m_2))$. 

\smallskip
Conversely, suppose that $m_1$, $m_2\in M$ are such that $p'(m_1)\ne p'(m_2)$. Then, there exists some $y:(M,\sigma_M)
\longrightarrow (Y,\sigma_Y)$ with $Ker(p)\subseteq Ker(y)$ such that $y(m_1)\ne y(m_2)$. From the construction in Lemma \ref{Lq10.3}, we see that $(Y,\sigma_Y)$ may be assumed to be injective in $Heymod_H^{\mathfrak s}$. From Proposition \ref{P8.6}
and Proposition \ref{Pt9.4}, it is clear that a normal monomorphism in $Heymod_H^{\mathfrak s}$ is also a monomorphism in 
$Heymod_H^{\mathfrak s}$. Hence, $y:(M,\sigma_M)
\longrightarrow (Y,\sigma_Y)$ extends to some $z:(N,\sigma_N)\longrightarrow (Y,\sigma_Y)$ such that $z\circ i=y$. It follows that
$z(i(m_1))\ne z(i(m_2))$. 

\smallskip
We now claim that $Ker(z)\supseteq Ker(q)$. Indeed, if $n\in N$ is such that $n\in Ker(q)$, we know from the assumption $Ker(q)\subseteq (M,\sigma_M)$ 
that $n=i(m)$ for some $m\in M$. Then, $m\in Ker(q\circ i)$. But $Ker(q\circ i)=Ker(j\circ p)=Ker(p)\subseteq Ker(y)$. Hence, 
$m\in Ker(y)$. Since $y=z\circ i$, this shows that $n=i(m)\in Ker(z)$. 

\smallskip
Since $q$ is a normal epimorphism and $Ker(z)\supseteq Ker(q)$, the fact that $z(i(m_1))\ne z(i(m_2))$ implies that
$q(i(m_1))\ne q(i(m_2))$. Since $j$ is an injective map, we now have an equivalence
\begin{equation}\label{leqbeq}
p'(m_1)=p'(m_2)\textrm{ }\Leftrightarrow\textrm{ }q(i(m_1))=q(i(m_2))\textrm{ }\Leftrightarrow\textrm{ }p(m_1)=p(m_2)
\end{equation} for all $m_1$, $m_2\in M$. Since $p$ and $p'$ are both surjective, it is clear from \eqref{leqbeq} that 
$p=p'$. 
\end{proof}

\begin{Thm}\label{Tq10.10} Let $H$ be a finite Boolean algebra. Then, $Heymod_H^{\mathfrak s}$ is a semiexact homological category. 

\end{Thm}

\begin{proof}
We know from Theorem \ref{Tt9.6} that $Heymod_H^{\mathfrak s}$ is a semiexact category. The result now follows from Definition \ref{Dq10.625} along
with Propositions \ref{Pq10.4}, \ref{Pq10.6} and \ref{Pq10.7}
\end{proof}

\section{Spectral spaces and Heyting submodules}

\smallskip
In this final section, we let $H$ be an arbitrary (not necessarily finite) Heyting algebra. We will show that the collection of Heyting submodules
of a given Heyting module $M$ can be given the structure of a spectral space. We will do this by using a criterion of Finocchiaro \cite{Fino}  in a manner
similar to \cite{FFS}, where it was shown that the collection of submodules of a given module over a commutative ring forms a spectral space. In \cite{AB},
it was shown that these techniques apply more generally to abelian categories that satisfy the (AB5) axiom (see also \cite{AB1}, \cite{Ray}).

\smallskip
We recall that a topological space is said to be spectral if it is homeomorphic to the Zariski spectrum of a commutative ring. A famous result of Hochster \cite{Hoch} shows
that a topological space $X$ is spectral if and only if satisfies: (a) $X$ is quasi-compact (b) the quasi-compact opens in $X$ are closed under
intersection and form a basis (c) every non-empty irreducible closed subset has a unique generic point. In other words, the property of being
spectral can be characterized in purely topological terms, without any reference to commutative rings. 

\smallskip
For a Heyting module $M$ over the given Heyting algebra $H$, we denote by $Sub(M)$ the collection of Heyting submodules of $M$. From Proposition 
\ref{P4.5} and
Proposition \ref{Pp3.65}, we know that a Heyting submodule $N\in Sub(M)$ is simply a distributive submodule of $M$ over the lattice $H$. For any finite collection of elements $\{m_1,...,m_n\}\in M$, we set
\begin{equation}\label{closed}
V(m_1,...,m_n):=\{\mbox{$N\in Sub(M)$ $\vert$ $m_1,...,m_n\in N$}\}
\end{equation} We let the $V(m_1,...,m_n)$ be a subbasis of closed sets for the topology on $Sub(M)$. In other words, a subbasis
of open sets for the topology on $Sub(M)$ is given by subsets of the form
\begin{equation}\label{open}
D(m_1,...,m_n):=Sub(M)\backslash V(m_1,...,m_n)
\end{equation} We will now show that this topology makes $Sub(M)$ into a spectral space. We recall here that a filter $\mathfrak F$ on a set $S$
is a collection of subsets of $S$ such that (a) $\phi\notin S$, (b) $Y$, $Z\in \mathfrak F$ $\Rightarrow$ $Y\cap Z\in \mathfrak F$ and (c) $Y\subseteq
Z\subseteq S$ and $Y\in \mathfrak F$ implies $Z\in \mathfrak F$. An ultrafilter on $S$ (see, for instance, \cite[$\S$ 1]{Fino}) is a maximal element in the collection 
of filters on $S$ ordered by inclusion. In particular, if $\mathfrak F$ is an ultrafilter, then for any subset $T\subseteq S$, exactly one
of $T$ and $(S\backslash T)$ lies in $\mathfrak F$. 

\begin{thm}\label{P10.1}
Let $H$ be a Heyting algebra and $M$ be a Heyting module over $H$. Then, the collection $Sub(M)$ of Heyting submodules of $M$
is a spectral space having the collection $\mathcal S$ of subsets of the form $D(m_1,...,m_n)$ as a subbasis of quasi-compact open sets. 
\end{thm}

\begin{proof}
We consider  $N$, $N'\in Sub(M)$ with $N\ne N'$. Then, we can pick some $m\in M$ such that $m$ lies in exactly
one of the two submodules $N$, $N'$. Then, $V(m)$ is a closed subset of $Sub(M)$ containing exactly one of the two points
$N$, $N'\in Sub(M)$. Hence, $Sub(M)$ is a $T_0$-space.

\smallskip
We now consider an ultrafilter $\mathfrak F$ on $Sub(M)$ and set 
\begin{equation}\label{ultraf}
N(\mathfrak F):=\{\mbox{$m\in M$ $\vert $ $V(m)\in \mathfrak F$}\}
\end{equation} We claim that $N(\mathfrak F)$ is a Heyting submodule of $M$. By Proposition \ref{Pp3.65}, we need to check that 
it is a distributive submodule of $M$. We consider therefore $m_1$, $m_2\in N(\mathfrak F)$ and some $c\in H$. Then, 
$V(m_1)$, $V(m_2)\in \mathfrak F$. From the definition 
of a filter and from \eqref{closed},  we obtain
\begin{equation}\label{eq10.4c}
V(m_1\vee m_2)\supseteq V(m_1)\cap V(m_2) \textrm{ }\Rightarrow\textrm{ }V(m_1\vee m_2)\in \mathfrak F \qquad
V(c\wedge m_1)\supseteq V(m_1) \textrm{ }\Rightarrow\textrm{ }V(c\wedge m_1)\in \mathfrak F
\end{equation} Hence, $N(\mathfrak F)\in Sub(M)$. We now claim that
\begin{equation}\label{fineq}
D(m_1,...,m_n)\in \mathfrak F \qquad\Leftrightarrow\qquad N(\mathfrak F)\in D(m_1,...,m_n)
\end{equation} First, we suppose that $N(\mathfrak F)\in D(m_1,...,m_n)$, i.e., there is some $m_k\in \{m_1,...,m_n\}$ such that $m_k\notin
N(\mathfrak F)$. Applying \eqref{ultraf},  we get $V(m_k)\notin \mathfrak F$. Since $\mathfrak F$ is an ultrafilter, this means that
the complement $D(m_k)\in \mathfrak F$. It is clear that $D(m_k)\subseteq D(m_1,...,m_n)$ and $\mathfrak F$ being a filter,
this means that $D(m_1,...,m_n)\in \mathfrak F$. 

\smallskip
On the other hand, if $N(\mathfrak F)\notin D(m_1,...,m_n)$, then $N(\mathfrak F)\in V(m_1,...,m_n)$. Hence,
$V(m_i)\in \mathfrak F$ for each $1\leq i\leq n$. Since $\mathfrak F$ is a filter, we get $V(m_1,...,m_n)=\underset{i=1}{\overset{n}{\bigcap}}V(m_i)\in \mathfrak F$. Hence, $D(m_1,...,m_n)\notin \mathfrak F$. This proves the equivalence in \eqref{fineq}. 

\smallskip
Since $Sub(M)$ is a $T_0$-space, it now follows from the criterion in \cite[Corollary 3.3]{Fino} that $Sub(M)$ is spectral with subsets of the form
$D(m_1,...,m_n)$ being a basis of quasi-compact opens. 
\end{proof}

We can refine the result of Proposition \ref{P10.1} by considering closure operators on Heyting submodules in a manner similar 
to \cite[$\S$ 3]{AB} and \cite[$\S$ 3]{FFS}. These are inspired by closure operations on ideals in commutative algebra, such as taking the radical
closure, the integral closure, plus closure or Frobenius closure in certain special classes of rings (see \cite{Epstein} for a detailed survey).

\smallskip
We will say that a  Heyting submodule $N\in Sub(M)$ is finitely generated if there is a finite collection $\{m_1,...,m_n\}$ of elements
of $M$ such that $N$ is the smallest Heyting submodule containing all of them. For $N\in Sub(M)$, we will denote by $fg(N)$ the collection
of finitely generated Heyting submodules of $N$. 

\begin{defn}\label{D10.2}
Let $H$ be a Heyting algebra and let $M$ be a Heyting module over $H$. A closure operation $\mathbf c$ on Heyting submodules
of $M$ is an operator $\mathbf c:Sub(M)\longrightarrow Sub(M)$ such that:

\smallskip
(a) $\mathbf c$ is extensive, i.e., $N\subseteq \mathbf c(N)$ for each $N\in Sub(M)$.

\smallskip
(b) $\mathbf c$ is order preserving, i.e.,  $N\subseteq N'$ for $N$, $N'\in Sub(M)$ implies $\mathbf c(N)\subseteq \mathbf c(N')$. 

\smallskip
(c) $\mathbf c$ is idempotent, i.e., $\mathbf c(\mathbf c(N))=\mathbf c(N)$ for each $N\in Sub(M)$. 

\smallskip
We will say that a closure operator $\mathbf c:Sub(M)\longrightarrow Sub(M)$  is of finite type if it satisfies $\mathbf c(N)=\underset{N'\in fg(N)}{\bigcup}
\mathbf c(N')$ for every $N\in Sub(M)$. 
\end{defn}

\begin{lem}\label{L10.2}
Let  $\mathbf c:Sub(M)\longrightarrow Sub(M)$ be a closure operator on Heyting submodules of $M$. Then, the operator defined by setting
\begin{equation}\label{clf}
\mathbf c_f(N):=\underset{N'\in fg(N)}{\bigcup}
\mathbf c(N')\qquad\forall\textrm{ }N\in Sub(M)
\end{equation} is a closure operator of finite type.
\end{lem}

\begin{proof}
It is evident that $\mathbf c_f$ is extensive and order-preserving. If $N\in Sub(M)$ is finitely generated, it is clear from \eqref{clf} that
$\mathbf c_f(N)=\mathbf c(N)$. Hence, $\mathbf c_f(N)=\underset{N'\in fg(N)}{\bigcup}
\mathbf c(N')=\underset{N'\in fg(N)}{\bigcup}
\mathbf c_f(N')$, i.e., $\mathbf c_f$ is of finite type. It remains to show that $\mathbf c_f$ is idempotent. For this, we consider some
$N''\in fg(\mathbf c_f(N))$ having a generating set $\{n_1,...,n_k\}$. For each $n_i$, we can choose $N'_i\in fg(N)$ such that $n_i\in \mathbf c(N_i')$. Then,
$N_1'\vee ... \vee N_k'$ is a finitely generated Heyting submodule of $N$ and $N''\subseteq \mathbf c(N_1'\vee ... \vee N_k')$.  Since $\mathbf c$
is idempotent, we get $\mathbf c(N'')\subseteq \mathbf c(N_1'\vee ... \vee N_k')$ for each $N''\in fg(\mathbf c_f(N))$. We now have
\begin{equation}\label{eq10.7}
\mathbf c_f(\mathbf c_f(N))=\underset{N''\in fg(\mathbf c_f(N))}{\bigcup}
\mathbf c(N'')\subseteq \underset{N'\in fg(N)}{\bigcup}
\mathbf c(N')=\mathbf c_f(N)
\end{equation} Since $\mathbf c_f$ is extensive, \eqref{eq10.7} implies that $\mathbf c_f(\mathbf c_f(N))=\mathbf c_f(N)$. This proves the result. 
\end{proof}

\begin{thm}\label{P10.4}
Let $H$ be a Heyting algebra and $M$ be a Heyting module over $H$. Let $\mathbf c:Sub(M)\longrightarrow Sub(M)$ be a closure
operator of finite type.Then, the collection $Sub^{\mathbf c}(M)$ of Heyting submodules of $M$ fixed by $\mathbf c$ 
is a spectral space having the collection $\mathcal S^{\mathbf c}$ of subsets of the form $D(m_1,...,m_n)\cap Sub^{\mathbf c}(M)$ as a subbasis of quasi-compact open sets. 
\end{thm}

\begin{proof}
Since $Sub^{\mathbf c}(M)$ is equipped with the subspace topology induced by $Sub(M)$, it must be a $T_0$-space. Given an ultrafilter $\mathfrak F$
on $Sub^{\mathbf c}(M)$, we now set
\begin{equation}\label{ultrafc}
N(\mathfrak F):=\{\mbox{$m\in M$ $\vert $ $V(m)\cap Sub^{\mathbf c}(M)\in \mathfrak F$}\}
\end{equation} The order relations between subsets in \eqref{eq10.4c} continue to hold when intersected with $Sub^{\mathbf c}(M)$ and hence
it follows that $N(\mathfrak F)\in Sub(M)$. We claim that $N(\mathfrak F)\in Sub^{\mathbf c}(M)$, i.e., it is fixed by $\mathbf c$. 

\smallskip
We consider an element $m\in \mathbf c(N(\mathfrak F))$. Since $\mathbf c$ is of finite type, there is a finitely generated submodule 
$N'\subseteq N(\mathfrak F)$ such that $m\in \mathbf c(N')$. Suppose that $N'$ is generated by $\{m_1,...,m_k\}$. Then, if $N''\in Sub^{\mathbf c}(M)$
is such that $N''\supseteq N'$, we obtain $N''=\mathbf c(N'')\supseteq \mathbf c(N')$, i.e., $m\in N''$. In other words, we have
$V(m_1,...,m_k)\cap Sub^{\mathbf c}(M)\subseteq V(m)\cap Sub^{\mathbf c}(M)$. Since $m_1$, $m_2$, ..., $m_k\in N'\subseteq N(\mathfrak F)$, we get
$V(m_i)\cap Sub^{\mathbf c}(M)\in \mathfrak F$ for each $1\leq i\leq k$. Then, $V(m)\cap Sub^{\mathbf c}(M)\supseteq V(m_1,...,m_k)\cap Sub^{\mathbf c}(M)=\underset{i=1}{\overset{k}{\bigcap}}(V(m_i)\cap Sub^{\mathbf c}(M))\in\mathfrak F$ and hence
$V(m)\cap Sub^{\mathbf c}(M)\in \mathfrak F$. By \eqref{ultrafc}, it follows that $m\in N(\mathfrak F)$, i.e.,  $N(\mathfrak F)\in Sub^{\mathbf c}(M)$.

\smallskip
In a manner similar to the proof of Proposition \ref{P10.1}, it may be verified that
\begin{equation}\label{finceq}
D(m_1,...,m_n)\cap Sub^{\mathbf c}(M)\in \mathfrak F \qquad\Leftrightarrow\qquad N(\mathfrak F)\in D(m_1,...,m_n)\cap Sub^{\mathbf c}(M)
\end{equation} Since $Sub^{\mathbf c}(M)$ is a $T_0$-space, it now follows from the criterion in \cite[Corollary 3.3]{Fino} that $Sub^{\mathbf c}(M)$
is a spectral space with  the collection $\mathcal S^{\mathbf c}$ as a subbasis of quasi-compact open sets. 

\end{proof}

For each Heyting submodule $N\subseteq M$, we now define
\begin{equation}\label{clsr}
\overline{N}:=\{\mbox{$m\in M$ $\vert$ $m\leq n$ for some $n\in N$}\}
\end{equation} It is clear that $\overline{N}$ is a hereditary submodule and in fact the smallest hereditary submodule of $M$ containing
$N$. 

\begin{cor}\label{C10.5} Let $H$ be a Heyting algebra and $M$ be a Heyting module over $H$. Then, the collection of hereditary submodules
of $M$, equipped with the subspace topology induced by $Sub(M)$, forms a spectral space. 

\end{cor}

\begin{proof} It is clear that the operation $N\mapsto \overline{N}$ in \eqref{clsr} is a closure operator of finite type. The result now follows
from Proposition \ref{P10.4}.

\end{proof}

\end{document}